\documentclass[ijoc,nonblindrev]{informs4}

\OneAndAHalfSpacedXI 

\usepackage{amsmath,amssymb,amsfonts}
\usepackage{bbm}
\usepackage{upgreek}
\usepackage{xcolor}
\usepackage{mathrsfs}
\usepackage{tikz}
    \usetikzlibrary{arrows}    
\usepackage{float}
\usepackage{algorithm}
\usepackage[noend]{algpseudocode}
\usepackage{comment}
\usepackage{subcaption}

\usepackage{natbib}
 \bibpunct[, ]{(}{)}{,}{a}{}{,}%
 \def\bibfont{\small}%

\usepackage{rotating}
\usepackage{fancyvrb}

\usepackage{hyperref}\hypersetup{
    colorlinks,
    linkcolor={red!50!black},
    citecolor={blue!50!black},
    urlcolor={blue!80!black}
}

\newcommand{\obj}{f}

\newcommand{\feasibleSet}{\mathcal{X}}
\newcommand{\xvec}{\ensuremath{x}}
\newcommand{\val}[1]{\nu_{#1}}
\newcommand{\len}[1]{l_{#1}}
\newcommand{\valvec}{\val{}}
\newcommand{\lenvec}{\len{}}

\newcommand{\numvars}{\ensuremath{n}}
\newcommand{\set}[1]{\left\{#1\right\}}

\newcommand{\dd}{\mathcal{D}}
\newcommand{\N}{\mathcal{N}}
\newcommand{\A}{\mathcal{A}}
\newcommand{\rootnode}{\mathbf{r}}
\newcommand{\terminalnode}{\mathbf{t}}
\newcommand{\head}[1]{h_{#1}}
\newcommand{\tail}[1]{t_{#1}}
\newcommand{\layer}[1]{\tau_{#1}}

\newcommand{\inarcs}[1]{\Gamma^-({#1})}
\newcommand{\outarcs}[1]{\Gamma^+({#1})}
\newcommand{\yvec}{\ensuremath{y}}

\newcommand{\avec}{a}
\newcommand{\Gmat}{\ensuremath{G}}
\newcommand{\gvec}{G}
\newcommand{\hvec}{d}

\newcommand{\state}[1]{\ensuremath{s}_{#1}}
\newcommand{\svec}{\ensuremath{s}}
\newcommand{\Vfun}{\ensuremath{V}}
\newcommand{\statespace}[1]{\mathcal{S}_{#1}}

\newcommand{\cost}{\ensuremath{c}}
\newcommand{\Amat}{\ensuremath{A}}
\newcommand{\Bmat}{\ensuremath{B}}
\newcommand{\bvec}{\ensuremath{b}}

\newcommand{\Nmat}{\ensuremath{N}}
\newcommand{\feasKKT}{\mathcal{Y}}
\newcommand{\pivec}{\pi}
\newcommand{\gammavec}{\gamma}

\newcommand{\projs}{\mathscr{P}}

\newcommand{\uvec}{\delta}
\newcommand{\Uset}{\Delta}
\newcommand{\sep}{\mathscr{S}}

\newcommand{\wvec}{\mathbf{w}}
\newcommand{\evec}{\mathbf{e}}
\newcommand{\vvec}{\mathbf{v}}

\newcommand{\E}{\mathcal{E}}

\newcommand{\U}{\mathcal{U}}
\newcommand{\V}{\mathcal{V}}

\newcommand{\sink}{\mathbf{t}}

\newcommand{\del}{\boldsymbol\delta}

\tikzstyle{zero arc} = [draw,dashed, line width=0.5pt,->]
\tikzstyle{one arc} = [draw,line width=0.5pt,->]
\tikzstyle{zero arc infeasible} = [zero arc,
preaction={
	draw,gray!30!white,-,
	double=gray!30!white,
	double distance=3pt,
}]
\tikzstyle{one arc infeasible} = [one arc,
preaction={
	draw,gray!30!white,-,
	double=gray!30!white,
	double distance=3pt,
}]

\TheoremsNumberedThrough     
\ECRepeatTheorems

\EquationsNumberedThrough    

\MANUSCRIPTNO{}

\begin{document}

\RUNAUTHOR{Lozano, Bergman, and Cire}

\RUNTITLE{Constrained Shortest-Path Reformulations via Decision Diagrams}

\TITLE{Constrained Shortest-Path Reformulations via Decision Diagrams for Structured Two-stage Optimization Problems}

\ARTICLEAUTHORS{%

\AUTHOR{Leonardo Lozano}
\AFF{Operations, Business Analytics \& Information Systems, University of Cincinnati, 2925 Campus Green Drive,\\Cincinnati, OH 45221\\ 
	\EMAIL{leolozano@uc.edu} \URL{}}
	
\AUTHOR{David Bergman}
\AFF{Department of Operations and Information Management, University of Connecticut, 2100 Hillside Rd, Storrs, CT 06268\\ \EMAIL{david.bergman@uconn.edu}
}

\AUTHOR{Andre A. Cire}
\AFF{Dept. of Management, University of Toronto Scarborough and Rotman School of Management \\ Toronto, Ontario M1C 1A4, Canada
\\ \EMAIL{andre.cire@rotman.utoronto.ca} 
    }

} 

\ABSTRACT{%
Many discrete optimization problems are amenable to constrained shortest-path reformulations in an extended network space, a technique that has been key in convexification, bound strengthening, and search. In this paper, we propose a constrained variant of these models for two challenging classes of discrete two-stage optimization problems, where traditional methods (e.g., dualize-and-combine) are not applicable compared to their continuous counterparts. Specifically, we propose a framework that models problems as decision diagrams and introduces side constraints either as linear inequalities in the underlying polyhedral representation, or as state variables in shortest-path dynamic programming models. For our first structured class, we investigate two-stage problems with interdiction constraints. We show that such constraints can be formulated as indicator functions in the arcs of the diagram, providing an alternative single-level reformulation of the problem via a network-flow representation. Our second structured class is classical robust optimization, where we leverage the decision diagram network to iteratively identify label variables, akin to an L-shaped method. We evaluate these strategies on a competitive project selection problem and the robust traveling salesperson with time windows, observing considerable improvements in computational efficiency as compared to general methods in the respective areas. 
}


\maketitle


\section{Introduction}

Given a graph $G$ equipped with arc lengths, the constrained shortest-path problem (CSP) asks for the shortest path between two vertices of $G$ that satisfies one or more side constraints, such as general arc-based budget restrictions or limits on the number of nodes traversed. CSPs are fundamental models in the operations literature, with direct applications both in practical transportation problems \citep{festa2015constrained}, as well as in a large array of solution methodologies for routing and multiobjective optimization \citep{Irnich2005}. Due to its importance, the study of scalable algorithms for CSPs and related problems is an active and large area of research (e.g., \citealt{cabrera2020}, \citealt{vera2021computing}, \citealt{kergosien2022efficient}). 

In this paper, we expand upon this concept and propose using CSPs as a modelling construct for structured classes of optimization problems. Our approach consists of excluding either challenging constraints or variables of the original model, and reformulating the simplified system as a shortest-path problem in an extended network space, where paths have a one-to-one mapping with potential solutions. Solving the original problem then reduces to finding a \textit{constrained} shortest path, where side constraints incorporate information of the missing constraints or variables that were originally excluded. In particular, we exploit network reformulations based on decision diagrams \citep{bergman2016}, extracted from the state-transition graph of recursive models for combinatorial optimization problems.

The distinct attribute of this methodology is that one can leverage the underlying network to derive new reformulations and more efficient algorithms. More precisely, we exploit the dual description of a constrained shortest-path problem as either a network-flow model (i.e., its \textit{polyhedral} perspective), or as a label-setting search (i.e., its \textit{dynamic programming} perspective), applying each based on the complexity of formulating the side constraints within the decision diagram. 

In particular, we develop this approach on two classes of challenging discrete optimization problems. The first refers to binary two-stage programs of the form
\begin{align*}    
    \max_{\xvec^L, \xvec^F \in \{0,1\}^n} 
        \bigg \{ 
            c_1^L \xvec^L + c_2^L \xvec^F 
                \colon 
                \Amat^L \xvec^L + \Bmat^L \xvec^F \le \bvec^L, 
                \xvec^F \in \argmax_{ \overline{\xvec}^F \in \{0,1\}^n } 
                \left\{ 
                    c^F \overline{x}^F \colon \Amat^F \, \overline{\xvec}^F \le \bvec^F,  \overline{\xvec}^F \le \mathbf{1} - \xvec^L
                \right\}
        \bigg \},
\end{align*}
where variables $x^F$ and $x^L$ formulate a leader's and a follower's decisions, respectively. The model above is prevalent in network interdiction problems \citep{MortonEtal07, cappanera2011, hemmati2014, lozano2017b}, minimum edge and vertex blocker problems \citep{bazgan2011, Mahdavi2014}, and other classes of adversarial settings \citep{costa2011, Caprara2016, zare2018}. Note that the above is not a zero-sum game since the leader and follower objectives are not necessarily opposite, resulting in a more general class of problems than traditional network interdiction. Our challenge stems from the combinatorial structure associated with the follower's variables, often preventing direct extensions of duality-based solution methodologies from continuous two-stage settings.

We show that, alternatively, the follower's subproblem can be formulated as a CSP parameterized by the leader's variables $x^L$. In particular, the coefficient matrix of the CSP has a totally unimodular structure, and typical polyhedral approaches (i.e., dualize-and-combine) can be applied to reformulate the follower's subproblem as a feasibility system. Further, we also exploit the property of the duals of the follower's CSP reformulation to provide a novel convexification of the typical nonlinear inequalities from the dualize-and-combine technique. The final model is an extended, single-level mixed-integer linear programming (MILP) reformulation of the original two-stage problem, and thus amenable to commercial state-of-the-art solvers. 

The second problem class we investigate is classical robust optimization, 
\begin{align*}
    \min_{\xvec \in \{0,1\}}  \left\{ cx \colon \Amat(\delta) \, \xvec \le \bvec(\delta), \, \forall \delta \in \Uset \right \}.
\end{align*}
where $\Uset$ is a uncertainty set that parameterizes the realizations of the coefficient matrix $\Amat(\delta)$ and right-hand side $\bvec(\delta)$. Existing approaches are primarily based on cutting-plane algorithms akin to Benders decomposition \citep{ mutapcic2009oracles,Zeng13,ben2015oracle,ho2018oracles,BorreroLozano2020}. Each step in such procedures solves an MILP of increasing size, which may often inhibit the scalability of the approach.

We demonstrate that the robust problem can be reformulated as a CSP where side constraints are parameterized by elements $\delta$ of the uncertainty set $\Uset$. The resulting model, however, can be (potentially infinitely) large. We propose a methodology that starts with a smaller subset $\Uset' \subset \Uset$ and encode the resulting CSP problem as a dynamic program, thereby amenable to scalable combinatorial CSP labeling algorithms \citep{cabrera2020,vera2021computing}. Violated constraints are then identified through a separation oracle to augment $\Uset'$, and the procedure is repeated. Such a procedure is akin to a cutting-plane method, but where labels in the CSP play the role of Benders cuts in the proposed state-augmenting procedure. 

We evaluate our methodologies numerically on case studies in competitive project selection and the robust traveling salesperson problem, respectively. For the first case, we compare the proposed single-level reformulation with respect to a state-of-the-art generic bilevel solver based on branch-and-cut \citep{fischetti2017new}. Results suggest that, even for large networks and using default solver settings, our methodology provides considerable solution time improvements with respect to the branch-and-cut solver. For the robust case, we compare our method with respect to an MILP-based cutting plane method. We observed similar compelling improvement, where gains were primarily obtained when solving the CSP via an existing labeling method \citep{lozano2013} as opposed to a linear formulation with integer programming. 

\smallskip \textit{Contributions.} Our primary contribution is the development of a CSP-based reduction methodology. We show that one can either exploit network flow models or dynamic programming to obtain, respectively, extended mixed-integer linear formulations or algorithms that are built on solvers that are constantly evolving. The techniques are applied to two structured classes of challenging discrete optimization problems, which in our experiments have suggested runtime improvements whenever compact networks for the underlying combinatorial problem sizes are available. We also discuss generalizations of the framework for other classes of two-stage problems.

\smallskip \textit{Paper organization.} The paper is organized as follows. Section \ref{sec:related} presents a literature review on related algorithms. In Section \ref{sec:preliminaries}, we briefly review reformulations of discrete optimization problems as network models, more specifically decision diagrams. Section \ref{sec:CSPModels} formalizes the description of our proposed constrained path formulations over decision diagrams. Section \ref{sec:bilevel} implements the approach for a class of discrete two-stage programs and depicts an application on competitive project selection. Section \ref{sec:robustoptimization} considers the approach for classical robust optimization with an application to the robust traveling salesperson problem with time windows. Finally, Section \ref{sec:conclusion} concludes the manuscript and discusses future directions. Proofs are included in the online supplement.

\section{Related Work}
\label{sec:related}

Constrained shortest-path problems (CSP) define an extensive literature in both operations and the computer science literature. We refer to the survey by \cite{festa2015constrained} for examples of techniques and applications. In particular, the closest variant to our work refers to the shortest path with resource constraints (SPPRC), first proposed by \cite{universite1986fabrication} and widely investigated as a pricing model in column-generation approaches \citep{Irnich2005}. The classic SPPRC views constraints as limited ``resources'' that accumulate linearly over arcs as a path is traversed. Our work considers equivalent resource constraints over graph-based reformulations of general discrete optimization, specifically via decision diagrams in our context.

We consider two CSP formulations as the basis of our approaches. The first augments a classical reformulation of an integer program as a shortest-path problem in a given graph. Such reformulations dates back to the seminal work by \cite{gilmore1966theory} on the representation of knapsack problems via value functions, which provide polyhedral subadditive duals to integer programs. The dual of these systems are weighted shortest-path problems; see, e.g., \cite{wolsey1999integer}, II.4. This perspective is analogous to network extended formulations based on dynamic programming (e.g., \citealt{eppen1987solving, conforti2010extended, de2022arc}). A decision diagram analogously provides such shortest-path reformulations, with additional reduction techniques to compress the network size by exploiting symmetry that is not captured by the value function state \citep{hooker2013decision}. Other works modeling integer programs as decision diagrams investigate cut-generation procedures \citep{davarnia2020outer, castro2020combinatorial}, discrete relaxations \citep{van2022}, convexification of nonlinear constraints \citep{bergman2018discrete, bergman2021decision}, and strenghtening of big-M constraints in routing problems \citep{cire2019network, castro2020mdd}. In this work, we modify this reformulation by incorporating side constraints to the shortest-path problem, specifically to address the difficult structure of the studied two-stage problems. Specifically, our approach considers a dualize-and-combine technique that uses complementary slackness to reformulate an optimization problem into a feasibility one (see, e.g., survey by \citealt{smith2020survey}). However, we leverage the polyhedral subadditive dual given by a decision diagram to obtain an alternative convexification and single-stage reformulations to these problems.

The second CSP reformulation we consider is based directly on dynamic programming (DP), which have been used extensively in state-of-the-art techniques for large-scale CSPs (e.g., \citealt{dumitrescu2003improved,cabrera2020,vera2021computing}). In general, the DP in this context solves a shortest-path problem on an extended graph via a specialized recursion that \textit{labels} the states achievable at a node. The benefit of such models is that they do not require linearity and exploit the combinatorial structure of the graph for efficiency. We expand on such models in \S\ref{sec:CSPModels} and refer to \cite{Irnich2005} for general survey of labeling algorithms for CSPs and their effectiveness.

One of our problem settings can be seen as a special class of discrete two-stage optimization problems with so-called interdiction constraints. This defines an active research field with applications in energy and natural gas market regulation \citep{DempeEtal11,kalashnikov10}, waste management \citep{Xu12b}, bioengineering \citep{Burgard03}, and traffic systems \citep{Brotcorne01,Dempe12}, to name a few. Solution approaches include branch and bound \citep{DeNegreRalphs09,XuWang14}, Benders decomposition \citep{bolusani2022framework}, parametric programming \citep{Dominguez2010}, column-and-rown generation \citep{baggio2021}, and cutting-plane approaches based on an optimal value-function reformulations \citep{Mitsos10, Lozano2017}. The state-of-the-art general approach for such problems is a branch-and-cut algorithm for mixed-integer linear bilevel programs \citep{fischetti2017new}, which we consider as a benchmark in our computational runs. 




Robust optimization is also pervasive, with extensive literature in both theory \citep{bertsimas2004robust,ben2006extending,bertsimas2009constructing,li2011comparative,bertsimas2016reformulation} and applications~\citep{lin2004new,ben2005retailer,bertsimas2006robust,yao2009evacuation,ben2011robust,gregory2011robust,moon2011robust,gorissen2015practical,xiong2017distributionally}. For settings where variables are discrete, state-of-the-art algorithms first relax the model by considering only a subset of realizations, iteratively adding violated variables and constraints from missing realizations until convergence. This primarily includes cutting plane algorithms based on Benders decomposition \citep{ mutapcic2009oracles,Zeng13,ben2015oracle,ho2018oracles,BorreroLozano2020}. Such methodologies exploit, e.g., the structure of the uncertainty set to quickly identify violated constraints via a ``pessimization'' oracle that finds the worst-case realization of the uncertainty set.

In the robust setting, our methodology can alternatively be cast as a type of cutting-plane approach that augments an uncertainty set initially composed of a small number of realizations. This approach incorporates the scenario-specific constraints as resources in a full CSP reformulation of the robust problem based on decision diagrams, which we later remodel as a infinite-dimensional dynamic program.  We address such a program by adding missing states iteratively, where the new state captures one or more violated constraints of the original model. Existing state-augmenting algorithms have been applied by \cite{boland2006accelerated} in labeling methods to solve the standard CSP, as well as a in dynamic programs for stochastic inventory management \citep{rossi2011state}. 

\section{Preliminaries}
\label{sec:preliminaries}

We introduce in this section the concept of decision-diagram (DD) reformulations that we leverage throughout this work. We start in \S\ref{sec:networkDDs} by introducing the network structure and notation. Next, we briefly discuss in \S\ref{sec:DDcompilation} existing construction strategies and problems amenable to this model. 

\subsection{Network Reformulations via Decision Diagrams}
\label{sec:networkDDs}

In this context, a decision diagram $\dd$ is a network reformulation of the problem
\begin{align}
    \tag{DO} \label{model:discretep} 
    \min_{\xvec} \left\{ f(\xvec) \colon \xvec \in \feasibleSet \right\},
\end{align}
where $\feasibleSet \subseteq \mathbb{Z}^n$ is an $n$-dimensional finite set for $n \ge 1$. Specifically, $\dd = (\N,\A,\valvec,\lenvec)$ is an acyclic graph with node set $\N$ and arc set $\A \subseteq \N \times \N$. The set $\N$ is partitioned by $n+1$ layers $\N_1, \dots, \N_{n+1}$, where $\N_1 = \{\rootnode\}$ for a root node $\rootnode$ and $\N_{n+1} = \{ \terminalnode \}$ for a terminal node $\terminalnode$. Each arc $a \in \A$ is equipped with a value assignment $\val{a} \in \mathbb{Z}$ and a length $\len{a} \in \mathbb{R}$. Further, arcs only connect nodes in adjacent layers, i.e., with each $a \in \A$ we associate a tail node $\tail{a} \in \N_j$ and a head node $\head{a} \in \N_{j+1}$, $j \in \{1,\dots,n\}$. We denote by 
$\layer{a} \in \{1,\dots,n\}$ the layer that includes the tail node of $a$, i.e., $\tail{a} \in \N_{\layer{a}}$.

The decision diagram $\dd$ models \ref{model:discretep} through its $\rootnode-\terminalnode$ paths and path lengths.
Given an arc-specified $\rootnode-\terminalnode$ path $(a_1, \dots, a_n)$, where $\head{a_i} = \tail{a_{j+1}}$ for $j=1,\dots,n-1$, the solution $\xvec = (\val{a_1}, \dots, \val{a_n})$ composed of the ordered arc-value assignments is such that $\xvec \in \feasibleSet$ and the path length satisfies $\sum_{j=1}^n \len{a_j} = \obj(\xvec)$. Conversely, every solution $\xvec \in \feasibleSet$ maps to a corresponding $\rootnode-\terminalnode$ where its length matches the objective evaluation of $\xvec$. 

Thus, by construction, the shortest path in $\dd$ yields an optimal solution to \ref{model:discretep}. We also remark that if \ref{model:discretep} were a maximization problem, the longest $\rootnode-\terminalnode$ path provides instead the optimal solution, which is also computable efficiently in the size of $\dd$ since the network is acyclic.

\medskip
\begin{example}
    \label{ex:dd}
    Figure \ref{fig:knap_dd}(a) depicts a reduced decision diagram for the knapsack problem 
    $\max_{\xvec \in \feasibleSet} \left\{ 4x_1 + 3x_2 + 7x_3 + 8x_4 \colon \xvec \in \feasibleSet \right\}$   
    with  $\feasibleSet = \{ \xvec \in \{0,1\}^4 \colon 7 x_1 + 5 x_2 + 4 x_3 + x_4 \leq 8 \}$ (from \citealt{castro2020combinatorial}). Since variables are binaries, the value assignment of an arc $a \in \A$ is either $\val{a} = 0$ (dashed lines) or $\val{a} = 1$ (solid lines). An arc $a$ emanating from the $j$-th layer, i.e., $\tail{a} \in \N_j$, corresponds to an assignment $x_j = \val{a}$.  Further, the length of $a$ is either $0$ if $\val{a} = 0$ or the coefficient of $x_j$ in the objective otherwise. In particular, the longest path (in blue) provides the optimal solution $\xvec = (0,0,1,1)$ with objective $15$. \hfill $\square$

    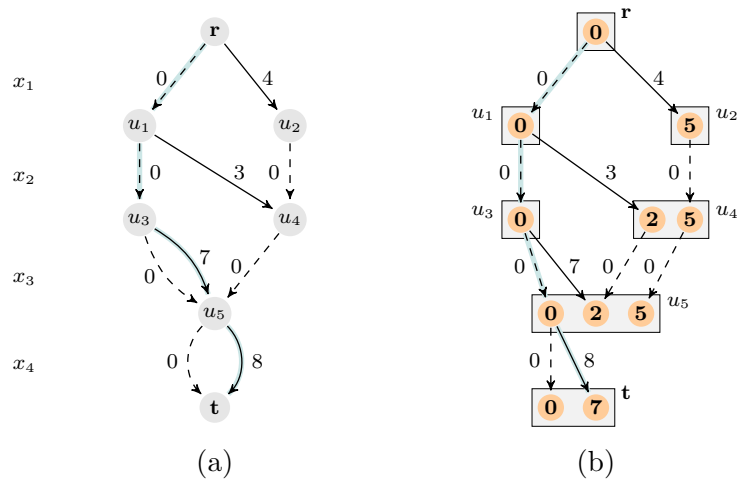
\begin{figure}[h]
        \begin{center}
            \tikzstyle{main node} = [circle,fill=gray!50,font=\scriptsize, inner sep=1pt]            
            \tikzstyle{text node} = [font=\scriptsize]
            \tikzstyle{arc text} = [font=\scriptsize]

            \tikzstyle{optimal arc} = [-,line width=2pt, color=teal!20!white]
            \tikzstyle{blocked arc} = [-,line width=2pt, color=red]

            \begin{tikzpicture}[->,>=stealth',shorten >=1pt,auto,node distance=1cm,
                thick]        
                \node[text node] (r) at (0,0) {$\quad$};                
                \node[text node] (l1) at (0,-0.5) {$x_1$};
                \node[text node] (l2) at (0,-1.75) {$x_2$};
                \node[text node] (l4) at (0,-3.1) {$x_3$};
                \node[text node] (l5) at (0,-4.25) {$x_4$};
                \node[text node] (t) at (0,-4.1) {$\quad$};
                \node (type) at (0,-5.75) {$\;$};                
            \end{tikzpicture}
            \hspace{2em}
            \begin{tikzpicture}[->,>=stealth',shorten >=1pt,auto,node distance=1cm,
                thick]        
                \tikzstyle{node label} = [circle,fill=gray!20,font=\scriptsize, inner sep=1pt]

                \node[node label] (r) at (0,0) {$\;\rootnode\;$};
                \node[node label] (u1)  at (-1,-1.25)  {$u_1$};
                \node[node label] (u2)  at (1,-1.25)  {$u_2$};
                \node[node label] (u3)  at (-1,-2.5) {$u_3$};
                \node[node label] (u4)  at (1,-2.5) {$u_4$};
                \node[node label] (u5)  at (0,-3.75) {$u_5$};
                \node[node label] (t) at (0,-5) {$\;\terminalnode\;$};
                \node (type) at (0,-5.75) {(a)};
                
                \path[every node/.style={font=\sffamily\small}]
                (r) 
                    edge[optimal arc] (u1)
                    edge[zero arc] node [left, arc text] {$0$} (u1)
                    edge[one arc] node [right, arc text] {$4$} (u2)
                (u1) 
                    edge[optimal arc] (u3)
                    edge[zero arc] node [right, arc text] {$0$} (u3)
                    edge[one arc] node [right, arc text] {$\;\;3$} (u4)
                (u2)
                    edge[zero arc] node [left, arc text] {$0$} (u4)
                (u3)
                    edge[zero arc, bend right=20] node [left, arc text] {$0$} (u5)

                    edge[optimal arc, bend left=20] (u5)
                    edge[one arc, bend left=20] node [right, arc text] {$7$} (u5)
                (u4)
                    edge[zero arc] node [left, arc text] {$0$} (u5)
                (u5)
                    edge[zero arc, bend right=45] node [left, arc text] {$0$} (t)
                    
                    edge[optimal arc, bend left=45] (t)
                    edge[one arc, bend left=45] node [right, arc text] {$8$} (t);                            
            \end{tikzpicture}        
            \quad\quad\quad\quad\quad
            \begin{tikzpicture}[->,>=stealth',shorten >=1pt,auto,node distance=1cm,
                thick]        
                \tikzstyle{main node} = [circle,fill=orange!40,font=\scriptsize, inner sep=1pt]
                
                \draw [fill=gray!10, ultra thin] (-0.25,-0.25) rectangle (0.25, 0.25);
                    \node[font=\scriptsize] at (0.4, 0.25) {$\rootnode$};                
                \draw [fill=gray!10, ultra thin] (-1.25,-1.50) rectangle (-0.75, -1);
                    \node[font=\scriptsize] at (-1.5, -1.1) {$u_1$};                
                \draw [fill=gray!10, ultra thin] (1,-1.50) rectangle (1.5, -1);
                    \node[font=\scriptsize] at (1.75, -1.1) {$u_2$};
                \draw [fill=gray!10, ultra thin] (-1.25,-2.25) rectangle (-0.75, -2.75);
                    \node[font=\scriptsize] at (-1.5, -2.4) {$u_3$};
                \draw [fill=gray!10, ultra thin] (0.5,-2.25) rectangle (1.50, -2.75);
                    \node[font=\scriptsize] at (1.75, -2.4) {$u_4$};
                \draw [fill=gray!10, ultra thin] (-0.85,-4) rectangle (0.85, -3.5);
                    \node[font=\scriptsize] at (1.1, -3.6) {$u_5$};
                \draw [fill=gray!10, ultra thin] (-0.85,-4.75) rectangle (0.25, -5.25);
                    \node[font=\scriptsize] at (0.4, -4.8) {$\terminalnode$};

                \node[main node] (r) at (0,0) {\textbf{0}};
                \node[main node] (u2a)  at (-1,-1.25)  {\textbf{0}};
                \node[main node] (u2b)  at (1.25,-1.25)  {\textbf{5}};
                \node[main node] (u3a)  at (-1,-2.5) {\textbf{0}};                
                \node[main node] (u4a)  at (0.75,-2.5) {\textbf{2}};
                \node[main node] (u4b)  at (1.25,-2.5) {\textbf{5}};
                \node[main node] (u5a)  at (-0.6,-3.75) {\textbf{0}};
                \node[main node] (u5b)  at (0,-3.75) {\textbf{2}};
                \node[main node] (u5c)  at (0.6,-3.75) {\textbf{5}};
                \node[main node] (t1)  at (-0.6,-5) {\textbf{0}};
                \node[main node] (t2)  at (0,-5) {\textbf{7}};

                \node (type) at (0,-5.75) {(b)};
                                
                \path[every node/.style={font=\sffamily\small}]
                (r) 
                    edge[optimal arc] (u2a)
                    edge[zero arc] node [left, arc text] {$0$} (u2a)
                    edge[one arc] node [right, arc text] {$4$} (u2b)
                (u2a)
                    edge[optimal arc] (u3a)                    
                    edge[zero arc] node [left, arc text] {$0$} (u3a)
                    edge[one arc] node [right, arc text] {$\;\;3$} (u4a)
                (u2b)                    
                    edge[zero arc] node [left, arc text] {$0$} (u4b)
                (u3a)
                    edge[optimal arc] (u5a)
                    edge[zero arc] node [left, arc text] {$0$} (u5a)
                    edge[one arc] node [right, arc text] {$7$} (u5b)
                (u4a)
                    edge[zero arc] node [left, arc text] {$0$} (u5b)
                (u4b)
                    edge[zero arc] node [left, arc text] {$0$} (u5c)
                (u5a)
                    edge[optimal arc] (t2)
                    edge[zero arc] node [left, arc text] {$0$} (t1)
                    edge[one arc] node [right, arc text] {$8$} (t2)
                ;
            \end{tikzpicture}        
        \end{center}
        \caption{(a) A decision diagram for the knapsack instance in Example \ref{ex:dd} and (b) the expanded network for Example \ref{ex:dp}. Green-shaded arcs represent the longest-path and optimal solution (in colour.)} 
        \label{fig:knap_dd} 
    \end{figure}
\end{example}

\subsection{Modeling Problems as DDs}
\label{sec:DDcompilation}

The standard methodology to model \ref{model:discretep} as $\dd$ consists of two steps \citep{bergman2016}. The first reformulates \ref{model:discretep} as a dynamic program (DP) and generates its state-transition graph, where states and actions are mapped to nodes and arcs, respectively. The second steps compresses the graph through a process known as \textit{reduction}, where nodes with redundant information are merged. This is a key process in the methodology, often reducing the graph by orders of magnitude in practice (e.g., \citealt{newton2019theoretical}). We detail the modeling framework in Appendix \ref{sec:modelDD} for reference.

Due to the intrinsic connection between DPs and decision diagrams, it often follows that decision diagrams are appropriate if \ref{model:discretep} has a compact recursive reformulation, in that the state space is of relatively low dimensionality either generally or for instances of interest. While the classical examples are knapsack problems with small right-hand sides, current studies expand on this class by either proposing more compact DP formulations or exploiting the structure of $\dd$ via reduction. Recent examples include submodular and nonlinear problem classes \citep{bergman2018discrete,davarnia2020outer}, semi-definite inequalities \citep{castro2020combinatorial}, scheduling \citep{cire2019network}, and graph-based problems with limited bandwidth \citep{haus2017compact}. In a two-stage stochastic programming setting, \cite{lozano2018binary} and \citet{macneil2024leveraging} also use decision diagrams as a convexification device for a Bender's decomposition approach. We refer to \cite{bergman2016} for additional examples.

\section{Constrained Shortest-Path Models}
\label{sec:CSPModels}

In this section, we introduce the two constrained shortest-path models over DDs that we leverage in our CSP reformulations. Specifically, we present a mixed-integer linear programming model (MILP) in \S\ref{sec:polyhedralperspective} and a dynamic program (DP) in \S\ref{sec:functionalperspective}. The first model exposes polyhedral structure and is appropriate, e.g., if the resulting formulation is sufficiently compact for mathematical programming solvers, which we leverage for our leader-follower two-stage problem (\S\ref{sec:bilevel}). The second model preserves the decision diagram structure and is amenable for combinatorial constrained shortest-path algorithms, which we exploit for the classical robust optimization approach (\S\ref{sec:robustoptimization}).

\subsection{General Formulation}
\label{sec:generalmodel}

Let $\dd = (\N,\A,\valvec,\lenvec)$ be a decision diagram encoding a given discrete optimization problem \ref{model:discretep}. We investigate reformulations of optimization problems based on the constrained shortest-path model 
\begin{subequations}
\begin{align}
    \min_{\yvec \in \{0,1\}^{|\A|}}
        &\quad \sum_{a \in \A} \len{a} y_{a} \label{model:csp} \tag{CSP-$\dd$}
        \\
    \text{s.t.} 
        &\quad 
            \textnormal{$\{a \in \A \colon y_a = 1\}$ is an $\rootnode-\terminalnode$ path in $\dd$,} \label{csp:pathct} \\
        &\quad 
            \Gmat \yvec \le \hvec, \label{csp:side}
\end{align}
\end{subequations}
where the inequality system \eqref{csp:side} corresponds to $m \ge 1$ arc-based side constraints for a non-negative coefficient matrix $\Gmat \in \mathbb{R}^{m \times |A|}$ and a non-negative right-hand side vector $\hvec \in \mathbb{R}^m$. 

We consider additional modeling assumptions when reformulating problems as \ref{model:csp}. First, $\dd$ is of computationally tractable size in the number of variables $n$ of the model \ref{model:discretep} it encodes (e.g., as in the cases in \S \ref{sec:DDcompilation}). Second, building a new decision diagram $\dd'$ that also satisfies a subset of constraints \eqref{csp:side} is not practical due to, e.g., its potentially large size. Finally, for simplicity we also assume that there exists at least one $\rootnode-\terminalnode$ path that satisfies \eqref{csp:side}. Results below can be directly adapted otherwise to detect infeasibility. 

Our goal is to study \ref{model:csp} as a modeling approach to reveal structural properties and derive alternative solution methods to broader problem classes. In particular, modeling problems as \ref{model:csp} consists of identifying which constraints should be incorporated within $\dd$, and which should be cast in the form of side constraints \eqref{csp:side} by choosing an appropriate $\Gmat$ and $\hvec$. 

We note that \ref{model:csp} is a special case of the resource-constrained shortest-path problem \citep{Irnich2005} on the structure imposed by a decision diagram. For instance, if variables are binaries, each node has at most two outgoing arcs. Nonetheless, \ref{model:csp} still preserves generality regardless of the size of the network, which is reflected in its computational complexity below.

\begin{proposition}
    \label{prop:complexity}
    \ref{model:csp} is (weakly) NP-hard even when $m=1$ and $\dd$ has one node per layer.
\end{proposition}

\subsection{Mathematical Program}
\label{sec:polyhedralperspective}

A classical reformulation for \ref{model:csp} is a mathematical program that rewrites the path constraints via balance-of-flow equalities. More precisely, let $\inarcs{u} = \{ a \in \A \colon \head{a} = u \}$ and $\outarcs{u} = \{ a \in \A \colon \tail{a} = u \}$ be the set of arcs directed in and out of the node $u$, respectively. Then, the program
\begin{subequations}
\begin{align}
    \min_{\yvec \in \{0,1\}^{|\A|}} 
        \tag{MILP-$\dd$}
        \label{model:cspmilp}
        &\quad
            \sum_{a \in \A} \len{a} y_{a} \\
    \text{s.t.} 
        &\quad 
            \sum_{a \in \outarcs{\rootnode}} y_a = 1, \label{SP1} \\
        &\quad 
            \sum_{a \in \Gamma^+(u)} y_a - \sum_{a \in \Gamma^-(u)} y_a = 0, 
                &&\forall u \in \N \setminus\set{\rootnode, \terminalnode}, \label{SP2} \\
        &\quad 
            \Gmat \yvec \le \hvec \label{SP3},
\end{align}
\end{subequations}

is a binary linear reformulation of \ref{model:csp} which is suitable, e.g., to standard MILP solvers.

The objective function of \ref{model:cspmilp}, alongside constraints \eqref{SP1}-\eqref{SP2}, is a linear extended formulation of \ref{model:discretep} in the space of $\yvec$ variables \citep{behle2007binary}. That is, there exists a one-to-one mapping between solutions $\yvec$ satisfying \eqref{SP1}-\eqref{SP2} and a feasible solution $\xvec \in \feasibleSet$, and the objective function value of $\yvec$ match that of $f(\xvec)$. Thus, \ref{model:cspmilp} provides a polyhedral representation of the constrained shortest-path problem. We also note that similar models combining a shortest-path linear formulation with side constraints appear, e.g., in classical extended formulations based on DPs (e.g., \citealt{eppen1987solving, conforti2010extended}). 

\subsection{Dynamic Programming}
\label{sec:functionalperspective}

An alternative common model is to reformulate \ref{model:csp} as a DP, which allows for combinatorial algorithms that exploit the network structure and scale to larger problem sizes. The underlying principle is to perceive each inequality in \eqref{csp:side} as a ``resource'' with limited capacity. DP approaches enumerate paths ensuring that the consumed resource, here modeled as \textit{state} variables, does not exceed the budget specified by the inequality right-hand sides. 

We introduce such a DP model for \ref{model:csp} defined over the structure of a decision diagram $\dd$. With each node $u \in \N$ we associate a function $\Vfun_u \colon \mathbb{R}^m \rightarrow \mathbb{R}$ that evaluates to the constrained shortest-path value from $u$ to the terminal $\terminalnode$. The function argument is a state vector $\svec = (\state{1}, \dots, \state{m}) \in \mathbb{R}^m$ where each $\state{i}$ represents the amount consumed of the $i$-th resource, $i \in \{1,\dots,m\}$. Let $\gvec_a$ be the vector defined by the $a$-th column of $\Gmat$, $a \in \A$. We write the Bellman equations
\begin{align}
    \Vfun_u(\svec) = 
    \left\{
    \begin{array}{ll}
        \min_{a \in \outarcs{u} \colon \svec + \gvec_a \le \hvec} \left \{ \len{a} + V_{\head{a}}(\svec + \gvec_a)  \right\},
            & \textnormal{if $u \neq \terminalnode$,}\\
        0,
            & \textnormal{otherwise.}
    \end{array}    
    \right.
    \tag{DP-$\dd$} \label{model:cspdp}
\end{align}
where $\min_{\emptyset} \{ \cdot \} = +\infty$ and $V_{\head{a}}$ is the value function of the head node $\head{a}$ of arc $a$. Note that \ref{model:cspdp} resembles a traditional shortest-path forward recursion in a graph \citep{cormen2009introduction}, except that traversals are constrained by the condition $\svec + \gvec_a \le \hvec$.

\begin{proposition}
    \label{prop:DPvalidity}
    The optimal value of \ref{model:csp} is $\Vfun_{\rootnode}(\mathbf{0})$, where $\mathbf{0} \equiv (0,\dots,0) \in \mathbb{R}^m$.
\end{proposition}

Insights into the distinction between \ref{model:cspmilp} and \ref{model:cspdp} are revealed by duality. Given Proposition \ref{prop:DPvalidity}, let $\statespace{u} \subseteq \mathbb{R}^m$ be the set of states that are reachable at a node $u \in \N$ by the recursion starting with $\Vfun_{\rootnode}(\mathbf{0})$ and budget $\hvec$, i.e., $\statespace{\rootnode} = \{\mathbf{0}\}$ and 
\begin{align*}
    \statespace{u} = \left\{ \svec' = \svec + \gvec_a \colon \svec' \le \hvec, \, \exists a \in \inarcs{u}, \, \exists \svec \in \statespace{\tail{a}} \right\}, &&\forall u \in \N \setminus \{\rootnode\}.
\end{align*}

The set $\statespace{u}$ is finite for all nodes $u$ because $\dd$ has a finite number of arcs. We can therefore reformulate \ref{model:cspdp} as a linear program with the evaluations of $\Vfun_u(\cdot)$ as variables \citep{bertsekas2012dynamic}:
\begin{subequations}
\begin{align}    
\max_{\mathbf{V}} 
    \quad& 
    \Vfun_{\rootnode}(\mathbf{0}) \\
\textnormal{s.t.}
    \quad&
        \Vfun_{\tail{a}}(\svec) \le \len{a} + \Vfun_{\head{a}}\left(\svec + \gvec_a\right), 
            && \forall a \in A, \svec \in \statespace{\tail{a}}, \label{lpdp1} \\
    \quad&
        \Vfun_{\terminalnode}(\svec) = 0, 
            && \forall \svec \in \statespace{\terminalnode}. \label{lpdp2}
\end{align}
\end{subequations}

Notice that, because of \eqref{lpdp2}, inequalities \eqref{lpdp1} simplify after replacing all variables $V_{\head{a}}(\cdot)$ by 0 whenever $\head{a} = \terminalnode$. The dual of the linear program obtained after such an adjustment is 
\begin{subequations}
\begin{align}
    \min_{w \ge 0}
        &\quad 
            \sum_{a \in A} \sum_{\svec \in \statespace{\tail{a}}} \len{a} w_{a, \svec} \label{model:dplp} \tag{LP-$\dd$}
        \\
    \text{s.t.} 
        &\quad 
            \sum_{a \in \outarcs{\rootnode}} w_{a, \mathbf{0}} = 1,  \\
        &\quad 
            \sum_{a \in \Gamma^+(u)} w_{a,\svec} - \sum_{a \in \Gamma^-(u)} \sum_{\substack{ \svec' \in \statespace{\tail{a}} \\ \svec' + \gvec_a = \svec} } w_{a, \svec'} = 0, 
                    &&\forall u \in \N \setminus \{ \rootnode, \terminalnode \}, \forall \svec \in \statespace{u}.
\end{align}
\end{subequations}

Thus, \ref{model:cspdp} solves a (non-constrained) shortest-path problem on an extended variable space $w_{a, \svec}$ that incorporates the state information $\svec$, as opposed to the DD arc-only based variables $y_{a}$ in \ref{model:cspmilp}. A solution to \ref{model:cspdp} is equivalently a state-arc trajectory 
$(\svec_1, a_1), (\svec_2, a_2), \dots, (\svec_n, a_n)$ 
such that $(a_1,\dots,a_n)$ is an $\rootnode-\terminalnode$ path in $\dd$ and $\svec_i = \svec_{i-1} + \gvec_{a_{i-1}} \le \hvec$ for $i=2,\dots,n$, $\svec_1 = \mathbf{0}$. 

\medskip
\begin{example}
    \label{ex:dp}
    Consider the knapsack instance from Example \ref{ex:dd} and the side constraint 
    \begin{align}
        \label{ex:side}
        5x_1 + 2x_2 + 2x_3 + 7x_3 \le 7.
    \end{align}

    We can formulate it as \ref{model:csp} using the DD from Figure \ref{fig:knap_dd}(a) to encode the original feasible set $\feasibleSet$, plus a single side constraint $\sum_{a \in \A} g_a y_a \le 7$ where $g_a$ is the coefficient of variable $x_{\layer{a}}$ in \eqref{ex:side} if $\val{a} = 1$, and $0$ otherwise. Figure~\ref{fig:knap_dd}(b) depicts the extended network space where \ref{model:dplp} is defined. The numbers on the node circles represent the reachable states (i.e., labels) and the squares the original nodes in Figure \ref{fig:knap_dd}(a) they are associated with. There exists a single label since $m=1$. 
    
    Further, two nodes are connected if the labels are consistent with the side constraint and the arc exists in the original DD. For example, the labels in $u_4$ only have outgoing arcs with value $0$, and their label remain constant since the coefficient of such arcs is also zero. \hfill $\square$
\end{example}

\medskip 

State-of-the-art combinatorial methods for 
general constrained path problems, such as the pulse method \citep{cabrera2020} and the CHD \citep{vera2021computing}, are directly applicable and operate on the network specified in \ref{model:dplp} by carefully choosing which states to enumerate, a process commonly known as labeling. Such methods are effective if such labels have exploitable structure, which we leverage in \S\ref{sec:robustoptimization}. Otherwise, if the set of states is difficult to describe but the representation is sufficiently compact, then \ref{model:cspmilp} could be more beneficial, which we investigate in \S \ref{sec:bilevel}.


\section{Leader-Follower Two-Stage Programs}
\label{sec:bilevel}

In this section, we introduce single-level reformulations based on \ref{model:csp} for two-stage programs
\begin{subequations}
\begin{align}
\label{model:bilevel}
\tag{DTS}
\max_{\xvec^L, \xvec^F} 
    &\quad
        \sum_{j=1}^n 
        \left( 
            \cost_{1j}^L \, x_{j}^L +  \cost_{2j}^L \, x_{j}^F
        \right) 
        \\
\textnormal{s.t.}
    &\quad
        \Amat^L \xvec^L + \Bmat^L \xvec^F \le \bvec^L, \\
    &\quad 
        \xvec^F \in \argmax_{ \overline{\xvec}^F }
        \left\{
            \sum^n_{j=1} \cost_{j}^F \, \overline{x}_j^F 
            \colon 
            \Amat^F \, \overline{\xvec}^F \le \bvec^F, \;\; 
            \overline{\xvec}^F \le \mathbf{1} - \xvec^L, \;\;
            \overline{\xvec}^F \in \{0,1\}^n
        \right\}, \label{bilevel:follower} \\
    &\quad 
    \xvec^L \in \{0,1\}^n,
\end{align}
\end{subequations}
where $\cost_1^L, \cost_2^L, \cost^F \in \mathbb{R}^{n}$ are cost vectors; $\Amat^L, \Bmat^L \in \mathbb{R}^{m_L \times n}$ are the \textit{leader's} coefficient matrices for $m_L \ge 0$; $\Amat^F \in \mathbb{R}^{m_F \times n}$ is the \textit{follower's} coefficient matrix for $m_F \ge 0$; $\bvec^L \in \mathbb{R}^{m_L}$, $\bvec^F \in \mathbb{R}^{m_F}$ are right-hand side vectors of the leader's and follower's, respectively; and $\mathbf{1}$ is an $n$-vector of ones. We assume the \emph{relatively complete recourse property}, i.e., for any feasible leader solution there exists a corresponding feasible follower response, ensuring the existence of an optimal solution for the follower problem described in \eqref{bilevel:follower}. This model is also \textit{optimistic}, in that it assumes the follower will pick its optimal solution that maximizes the leader's objective, in case of degeneracy.

Formulation \ref{model:bilevel} is prevalent in combinatorial applications from the literature. Specifically, the inequality $\overline{\xvec}^F \le \mathbf{1} - \xvec^L$ in \eqref{bilevel:follower} represents exogenous ``blocking'' decisions by a leader to prevent actions from an non-cooperative follower, which operates optimally according to its own utility function. Examples of problems modeled as \ref{model:bilevel} include network interdiction problems \citep{MortonEtal07, cappanera2011, hemmati2014, lozano2017b}, minimum edge and vertex blocker problems \citep{bazgan2011, Mahdavi2014}, and other classes of adversarial problems \citep{costa2011, Caprara2016, zare2018}.

Problems of this class, however, are notoriously difficult because of the ``argmax" constraint \eqref{bilevel:follower} that embeds the follower's optimization problem into the leader's model. In particular, the follower problem is non-convex since $\xvec^F$ are binaries, preventing the use of typical techniques for continuous problems such as dualize-and-combine \citep{kleinert2021survey}. Existing state-of-the-art methodologies rely, e.g., on specialized cutting-plane and search techniques \citep{fischetti2017new}.

In this section, we present a reformulation of \ref{model:bilevel} as \ref{model:cspmilp}, which is suitable to standard mathematical programming solvers if the underlying network $\dd$ is of tractable size. We start in \S\ref{sec:followerRemodel} by rewriting constraint \eqref{bilevel:follower} as a constrained shortest-path model, and extract an optimality certificate via polyhedral structure. Next, we present in \S\ref{sec:leaderRemodel} the full MILP reformulation based on the polyhedral structure of \ref{model:cspmilp}. Finally, we discuss in \S\ref{sec:bilevelgeneralizations} generalizations of the technique for broader classes of two-stage programs. 

\subsection{Reformulation of the Follower's Optimization Problem}
\label{sec:followerRemodel}

Let $\dd^F = (\N,\A,\valvec,\lenvec)$ encode the discrete optimization problem \ref{model:discretep} with feasible set 
\begin{align}
    \feasibleSet = 
    \{
        \xvec \in \{0, 1\}^n 
        \colon
        \Amat^F \xvec \le \bvec^F
    \} 
\end{align}
and objective $f(\xvec) = \sum_{j=1}^n \cost_j^F x_j$. That is, paths in $\dd^F$ correspond to solutions that satisfy the follower's constraints $\Amat^F \xvec \le \bvec^F$ and have arc lenghts given by
\begin{align*} 
\len{a}=&
\begin{cases}
\cost_{\layer{a}}^F, &\text{if } \val{a} = 1, \\
0, & \text{otherwise},
\end{cases} 
    \quad \quad \forall a\in\A,
\end{align*}
where we recall that $\val{a}$ is the value assignment of arc $a$. Proposition \ref{prop:followerDD} shows the link between constraint \eqref{bilevel:follower} and \ref{model:csp}.

\begin{proposition}
\label{prop:followerDD}
Let $\xvec^L \in \{0,1\}^n$ be any feasible leader's solution in \ref{model:bilevel}. There exists a one-to-one mapping between a follower's solution $\xvec^F$ satisfying \eqref{bilevel:follower} and an optimal solution of \ref{model:csp} defined over $\dd^F$ with a maximization objective sense and $m \leq |\A|$ side constraints of the form
\begin{align}
    y_a \le 1 - x^L_{\layer{a}}, \;\; \forall a \in \A \colon \val{a} = 1. \label{csp:follower}
\end{align} 
That is, the columns of $\Gmat$ associated with arcs $a \in \A$ such that $\val{a} = 1$ form an identity matrix, and the remaining columns are zero vectors. Further, $\hvec_{a} = 1 - x^L_{\layer{a}}$ if $\val{a} = 1$ and $\hvec_{a} = 0$ otherwise.
\end{proposition}

Proposition \ref{prop:followerDD} leads to a combinatorial description of solutions $x^F$ observing \eqref{bilevel:follower} as paths in $\dd^F$ subject to side constraints \eqref{csp:follower}. Thus, we can derive an equivalent mathematical program reformulation of constraint \eqref{bilevel:follower} through \ref{model:cspmilp}. Inequalities \eqref{csp:follower}, however, have a special form in that they only impose bounds on arc variables $\yvec$, not significantly changing the coefficient matrix structure. Based on this observation, we present a more general result applicable when side constraints \eqref{csp:side} do not break the integrality of the network-based extended model \eqref{SP1}-\eqref{SP2} presented in \S\ref{sec:polyhedralperspective}. We recall that $\gvec_a$ is the $a$-th row of $\Gmat$, $a \in \A$.

\begin{proposition}
    \label{prop:feasiblitySet}
    Let $\dd = (\N,\A,\valvec,\lenvec)$ encode a discrete optimization problem \ref{model:discretep} and let $\Nmat$ be the coefficient matrix of $\yvec$ associated with the linear system 
    \eqref{SP1}-\eqref{SP2}.  If the matrix 
    $\left[ \begin{array}{l} \Nmat \\ \Gmat \end{array} \right]$ 
    is totally unimodular and $\hvec$ is integral, then 
    the set of optimal solutions to \ref{model:cspmilp} is $\textnormal{Proj}_{\yvec} \feasKKT$, where 
    \begin{subequations}
    \begin{align}
        \feasKKT = 
        \big\{ (\yvec, \pivec, \gammavec) \in \mathbb{R}^{|A|} \times \mathbb{R}^{|\N|} \times \mathbb{R}^{|\A|}
        \colon 
            &\gammavec \ge 0, \;\; \eqref{SP1}-\eqref{SP3},
            \\
            &\pi_{\tail{a}} - \pi_{\head{a}} + \gammavec^{\top} \gvec_{a} \ge \len{a}, \;\;\; \forall a \in \A,
            \\
            &\lenvec^{\top} \yvec - \pi_{\rootnode} - \gammavec^{\top} \hvec = 0, \label{eq:kkt_ct3} \\
            &\yvec \in \{0,1\}^{|\A|} \label{eq:kkt_ct4}
            \big\}.
    \end{align}
    \end{subequations}
\end{proposition}

\smallskip 

Augmenting a network matrix with the identity matrix $G$ in this setting preserves total unimodularity (Prop.III-2.1, \citealt{wolsey1999integer}). Since $\hvec$ is integral, it follows from Proposition \ref{prop:feasiblitySet} that $\xvec^F$ satisfies \eqref{bilevel:follower} if and only if the continuous linear system \eqref{follow:ref1.1}-\eqref{follow:refVars} is feasible for a given leader solution $\xvec^L$:
\begin{subequations}
\begin{align}
 &\sum_{a \in \outarcs{\rootnode}} y_a = 1, \label{follow:ref1.1} \\
    &\sum_{a \in \Gamma^+(u)} y_a - \sum_{a \in \Gamma^-(u)} y_a = 0, &\forall u \in \N \setminus\set{\rootnode, \terminalnode}, \label{follow:ref1.2} \\
    & y_a \le 1 - x^L_{\layer{a}},  &\forall a \in \A \colon \val{a} = 1, \label{follow:ref1.3}\\
    &\pi_{\tail{a}} - \pi_{\head{a}} + \gamma_a \ge \len{a} \;\;\; &\forall a \in \A \colon \val{a} = 1, \label{follow:ref2} \\
    &\pi_{\tail{a}} - \pi_{\head{a}} \ge 0 \;\;\; &\forall a \in \A \colon \val{a} = 0, \label{follow:ref2.5} \\
    &\lenvec^{\top} \yvec - \pi_{\rootnode} - \sum_{a \in \A \colon \val{a} = 1} \gamma_a (1 - x^L_{\layer{a}}) = 0, \label{follow:ref3} \\
    &x^F_j - \sum_{a \in \A \colon \tail{a} \in \N_j, \val{a} = 1} y_a = 0, 
    &\forall j \in \{1,\dots,n\}, \label{follow:ref4} \\
    &x^F \in \{0,1\}^n, \yvec \in \{0,1\}^{|\A|},  \gammavec \ge 0, \label{follow:refVars} 
\end{align}
\end{subequations}

where $\pivec$ and $\gammavec$ are of appropriate dimensions in terms of $\feasKKT$. We note that the system \eqref{follow:ref1.1}-\eqref{follow:refVars} follows the structure of KKT-based reformulation approaches, i.e., \eqref{follow:ref1.1}-\eqref{follow:ref1.3}  ensure primal feasibility, \eqref{follow:ref2}-\eqref{follow:ref2.5} require dual feasibility, and \eqref{follow:ref3} ensures strong duality.

\subsection{Leader's MILP Model}
\label{sec:leaderRemodel}
An extended, single-level reformulation of \ref{model:bilevel} can be obtained directly when replacing \eqref{bilevel:follower} by \eqref{follow:ref1.1}-\eqref{follow:refVars}. However, the resulting formulation is nonlinear because of the terms $\gamma_a (1 - x^L_{\layer{a}})$ in \eqref{follow:ref3}. We show next how to linearize \eqref{follow:ref3} based on  Proposition \ref{prop:optSolFollower}.

\begin{proposition}
    \label{prop:optSolFollower}
    Any feasible solution $(\hat{\xvec}^F, \hat{\yvec}, \hat{\pivec}, \hat{\gammavec})$ to \eqref{follow:ref1.1}-\eqref{follow:refVars} can be modified to satisfy $\hat{\gamma}_a = M_a$ for all arcs $a \in \A$ such that $x^L_{\layer{a}} = 1$ and $\val{a} = 1$ preserving feasibility, where $M_{a}$ values are large constants that satisfy
\begin{equation}
 M_{a} \ge \len{a}+ \hat{\pi}_{\head{a}}-\hat{\pi}_{\tail{a}}  ,  \;\; \forall a \in \A\colon x^L_{\layer{a}} = 1 \text{ and } \val{a} = 1. \label{bigMCondition}
\end{equation} 
\end{proposition}

\medskip
Using the result above, we rewrite the left-hand side of \eqref{follow:ref3} as 
\begin{align*}
    \lenvec^{\top} \yvec - \pi_{\rootnode} - \sum_{a \in \A \colon \val{a} = 1 } \gamma_a (1 - x^L_{\layer{a}}) 
    &=
    \lenvec^{\top} \yvec - \pi_{\rootnode} - \sum_{a \in \A \colon \val{a} = 1 } \gamma_a + \sum_{a \in \A \colon \val{a} = 1 } \gamma_a x_{\layer{a}}^L \\
    &=
    \lenvec^{\top} \yvec - \pi_{\rootnode} - \sum_{a \in \A \colon \val{a} = 1 } \gamma_a + \sum_{a \in \A \colon \val{a} = 1 } M_a x^L_{\layer{a}},
\end{align*}
where the second equality follows from Proposition \ref{prop:optSolFollower}. This results in the following MILP reformulation: 
\begin{subequations}
\begin{align}
    \label{model:mipbilevel}
    \tag{MILP-DB}
    \max_{\xvec^L, \xvec^F, \yvec, \pivec, \gammavec} 
        &\quad
        \sum_{j=1}^n 
            \cost_{1j}^L \, x_{j}^L +  \cost_{2j}^L \, x_{j}^F
        \\
    \textnormal{s.t.}
        &\quad
            \Amat^L \xvec^L + \Bmat^L \xvec^F \le \bvec^L, \\
        &\quad 
            \eqref{follow:ref1.1}-\eqref{follow:ref2.5}, \eqref{follow:ref4}-\eqref{follow:refVars} \label{mipdp:primalDual} \\
        &\quad 
            \lenvec^{\top} \yvec - \pi_{\rootnode} - \sum_{a \in \A \colon \val{a} = 1 } \gamma_a + \sum_{a \in \A \colon \val{a} = 1 } M_a x^L_{\layer{a}} = 0, \label{mipdp:strongduality} \\
        &\quad 
            \gamma_a - M_a x^L_{\layer{a}} \ge 0, &&\forall a \in \A \colon \val{a} = 1, \label{mipdp:consistency} \\
        &\quad
            \xvec^L, \xvec^F \in \{0,1\}^n. \label{mipdp:vars}
\end{align}
\end{subequations}

The inequality \eqref{mipdp:consistency} ensures consistency of $\gammavec$ in line with Proposition \ref{prop:optSolFollower} and the derivation of \eqref{mipdp:strongduality}. We remark that \ref{model:mipbilevel} is a standard MILP. Thus, the formulation is appropriate to standard off-the-shelf solvers whenever $\dd^F$ is a suitable encoding of the follower's suproblem.

The model above contains big-M constants, which is typical in specialized two-stage techniques \citep{Wood93,CormicanEtal98,BrownEtal05a,LimSmith07,SmithEtal07,MortonEtal07,BayrakBailey08}. 

We show in Proposition \ref{prop:coolBigM} that under some mild conditions, a valid value for $M_a$ is given by the arc length $\len{a}$ for all $a \in \A \colon \val{a} = 1$. Noticeably, the relevance of this result is that it also leverages the decision diagram structure. We use the results of this proposition in the computations for the case study. 
\begin{proposition}
    \label{prop:coolBigM}
If the follower constraint coefficient matrix $\Amat^F \in \mathbb{R}^{m_F \times n}_+$, then setting $M_a = \len{a}$ for all $a \in \A \colon \val{a} = 1$ yields a valid formulation. 
\end{proposition}

Finally, Proposition \ref{prop:generalBigM} presents a general way of obtaining $M$-values for \ref{model:mipbilevel}, for settings where the non-negative assumption of Proposition \ref{prop:coolBigM} is violated.

\begin{proposition}
    \label{prop:generalBigM}
The following are valid values for the big-M constants in model \eqref{model:mipbilevel}:
\begin{equation}
M_a =  \len{a} + \sum_{j = 1}^n \max\{\cost_{j}^F , 0\} - \sum_{j = 1}^n \min\{\cost_{j}^F , 0\}, \quad \forall a\in\A.    
\end{equation}
\end{proposition}

\subsection{Case Study: Competitive Project Selection}
\label{sec:casestudy}

We present next a case study in competitive project selection. In this setting, two competing firms select projects to execute from a shared pool in a sequential manner, starting with the leader. Both firms seek to maximize their own profit, which is a function of both the projects as well as the competitor's selection. Applications of this setting appear in Marketing, where a company chooses advertisement actions that are also impacted by the actions of their competitors  \citep{denegre2011interdiction}.

Let $\projs = \{1,\dots,n\}$ be a set of $n$ projects in a pool shared by two firms. The leader firm considers both the net profit of projects executed, given by $c^L_j$, and a penalty for projects executed by the follower, denoted by $d^L_j$. The follower firm considers only the net profit of their executed projects, given by $c^F_j$. Similarly, the project costs and the firm budgets are $a^L_j,a^F_j$ and $b^L,b^F$, respectively. We write the competitive project selection problem as the following program
\begin{subequations}
\begin{align}
    \label{model:cpsp}
    \tag{CPSP}
    \max_{\xvec^L, \xvec^F} 
        &\quad
            \sum_{j \in \projs} 
            \left(
                c^L_j x^L_j - d^L_j x^F_j 
            \right)\\
    \textnormal{s.t.}
        &\quad
            \sum_{j \in \projs} a^L_j x^L_j \le b^L, \label{cpsp:budgetL} \\
        &\quad 
            \xvec^F \in \argmax_{ \overline{\xvec}^F }
            \left\{
                \sum_{j \in \projs} c^F_j \overline{x}^F_j 
                \colon 
                \sum_{j \in \projs} a^F_j \overline{x}^F_j \le b^F,
                \overline{\xvec}^F \le \mathbf{1} - \xvec^L, \;\;
                \overline{\xvec}^F \in \{0,1\}^n
            \right\}, \label{cpsp:follower} \\
        &\quad 
        \xvec^L \in \{0,1\}^n.
    \end{align}
\end{subequations}
The objective of \ref{model:cpsp} maximizes the leader's profit penalized by the follower's actions. The knapsack inequality \eqref{cpsp:budgetL} enforces the leader's budget. The follower's subproblem is represented by \eqref{cpsp:follower} and maximizes profit subject to an analogous budget constraint. Moreover, the constraint $\overline{\xvec}^F \le \mathbf{1} - \xvec^L$ imposes that only projects not selected by the leader can be picked. We remark that this problem is not a zero-sum game since the leader's objective does not necessarily minimize the follower's profit. Thus, \ref{model:cpsp} is not a knapsack interdiction problem but models a slightly more general setting in which the player's motivations are not in direct conflict.    

\subsubsection{Single-Level Reformulation.}

We construct $\dd^F = (\N,\A,\valvec,\lenvec)$ based on a standard DD formulation for the knapsack problem for which the state variable stores the amount of budget used after deciding if project $j \in \projs$ is selected or not (see e.g., \citet{bergman2016}). We generate $\dd^F$ from the state transition graph of the DP described above by removing nodes corresponding to the infeasible state and merging all terminal states. 

\begin{example}
    \label{ex:ddProjects}
Figure \ref{fig:ProjectSelection} shows a decision diagram generated for a follower's problem with $|\projs| =3$ having $a^F_1=a^F_2=2$, $a^F_3=4$, and $b^F = 5$. Solid lines represent yes-arcs, dashed lines represent no-arcs, and contributions to the follower's profit are shown alongside the yes-arcs.    
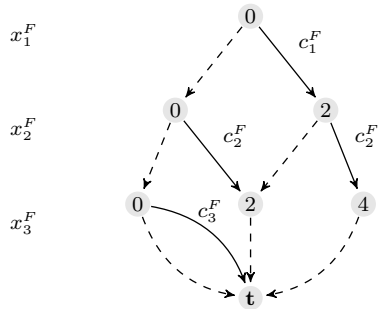
\begin{figure}[H]
	\begin{center}
	\usetikzlibrary{arrows}
	\tikzstyle{shaded} = [draw,line width=4pt,-,black!12]
	\tikzstyle{main node} = [circle,fill=gray!50,font=\scriptsize, inner sep=1pt]      \tikzstyle{text node} = [font=\scriptsize]
    \tikzstyle{arc text} = [font=\scriptsize]
	\begin{tikzpicture}[->,>=stealth',shorten >=1pt,auto,node distance=1cm,
                thick] 
	
	\node[text node] (r) at (0,0) {$\quad$};  
	\node[text node] (x^f_1) at (0,-0.25) {$x^F_1$};
	\node[text node] (x^f_2) at (0,-1.5) {$x^F_2$};
	\node[text node] (x^f_3) at (0,-2.75) {$x^F_3$};
	\node (type) at (0,-4) {$\;$};
	 \end{tikzpicture}
   \hspace{2em}
\begin{tikzpicture}[->,>=stealth',shorten >=1pt,auto,node distance=1cm,
                thick] 
\tikzstyle{node label} = [circle,fill=gray!20,font=\scriptsize, inner sep=1pt]
	\node [node label] (1) at (0,0) {$0$};
	\node [node label] (2) at (-1,-1.25) {$0$};
	\node [node label] (3) at (1,-1.25) {$2$};
	\node [node label] (4) at (-1.5,-2.5) {$0$};
	\node [node label] (5) at (0,-2.5) {$2$};
	\node [node label] (6) at (1.5,-2.5) {$4$};
	\node [node label] (t) at (0,-3.75) {$\sink$};
	\node (type) at (0,-4) {};

	\path[->](1) edge[one arc] node [above right, arc text] {$c^F_1$} (3);
	\path[->](1) edge [zero arc] node [right] {} (2);
	\path[->](2) edge[one arc] node [above right, arc text] {$c^F_2$} (5);
	\path[->](2) edge [zero arc] node [right] {} (4);
	\path[->](3) edge[zero arc] node [left] {} (5);
	\path[->](3) edge[one arc]  node [above right, arc text] {$c^F_2$} (6);

	\path[->](4) edge[one arc] [bend left] node [above, arc text] {$c^F_3$} (t);
	\path[->](4) edge [bend right, zero arc] node [below left] {} (t);
	\path[->](5) edge [zero arc] node [right] {} (t);
	\path[->](6) edge [bend left, zero arc] node [below left] {} (t);

	\end{tikzpicture}
	\end{center}
	\caption{A decision diagram corresponding to a follower's problem with 3 projects.}
	\label{fig:ProjectSelection}
\end{figure}
Each layer of the diagram corresponds to a follower's variable. There are 5 paths in the diagram that encode all the follower selections of projects that satisfy the budget constraint. \hfill $\square$
\end{example}

Following the approach described in \S\ref{sec:leaderRemodel}, we reformulate \ref{model:cpsp} over $\dd^F$ as the single-level MILP:
\begin{subequations}
\begin{align}
    \label{model:mipbilevelProjectSelection} \tag{CPSP-D}
    \max_{\xvec^L, \xvec^F, \yvec, \pivec, \gammavec} 
        &\quad
        \sum_{j \in \projs} 
                c^L_j x^L_j - d^L_j x^F_j \\
    \textnormal{s.t.}
        &\quad
            \sum_{j \in \projs} a^L_j x^L_j \le b^L, \label{cpspR:budgetL} \\
        &\quad 
            \eqref{mipdp:primalDual} - \eqref{mipdp:vars}. \label{cpspR:Reformulation} 
\end{align}
\end{subequations}
The objective function is the same as in the original formulation. Constraint \eqref{cpspR:budgetL} corresponds to the leader's budget constraint. Constraints \eqref{cpspR:Reformulation} remain exactly as defined in Section \ref{sec:leaderRemodel}, ensuring that $\yvec$ corresponds to a feasible path in $\dd^F$ (primal solution), that $(\pivec, \gammavec)$ is a feasible dual solution,  enforcing the relationship between flow-variables $\yvec$ and follower variables $\xvec^F$, and ensuring that $\yvec$ and $(\pivec, \gammavec)$ satisfy strong duality. Additionally, since all the project costs are nonnegative, we set $M_a = c^F_{\layer{a}}$ according to Proposition \ref{prop:coolBigM}.

\subsubsection{Numerical Analysis.}
\label{sec:bilevelnumericalanalysis} We compare \ref{model:mipbilevelProjectSelection}, dubbed DDR for decision-diagram reformulation, against an state-of-the-art approach for linear bilevel problems by \cite{fischetti2017new} based on branch and cut, denoted here by B\&C. We also experimented with the \emph{MibS} branch and cut approach by \cite{tahernejad2020branch} and found that B\&C performed much better for our instance set and thus decided to focus our comparison with B\&C. Our experiments use the original B\&C code provided by the authors. For consistency, DDR was implemented using the same solver as B\&C (ILOG CPLEX 12.7.1). All runs consider a time limit of one hour (3,600 seconds). We note that, while B\&C is based on multiple classes of valid inequalities and separation procedures, the DDR is a stand-alone model. We focus on a graphical description of the results in this sections; detailed tables and additional results are included in Appendix \ref{sec:Tables}. Source code will be available at the Github repository.

We generated instances of \ref{model:cpsp} based on two parameters: the number of items $n$, and the right-hand side tightness $t=\frac{b^L}{ \sum_{j \in \projs} a^L_j} = \frac{b^F}{ \sum_{j \in \projs} a^F_j}$, i.e., smaller values of $t$ correspond to instances with fewer feasible solutions. We considered five random instances for each $n \in \{30,40,50\}$ and $t \in \{0.10, 0.15, 0.20, 0.25\}$, generating 60 problems in total. Instances for each parameter configuration are generated as follows.  The coefficients $a^{L}$ are drawn uniformly at random from a discrete uniform distribution 
$\textnormal{U}(1,25)$ and then we set $a^{F} = a^{L}$. Budgets $b^{L} = b^{F}$ are set according to $t$. The project profits are proportional to their costs by setting $c^{L}_i = 5a^{L}_i+\xi^L_i$ and $c^{F}_i = 5a^{F}_i+\xi^F_i$, where $\xi$-values are drawn independently and uniformly at random from a discrete uniform distribution $\textnormal{U}(1,10)$.
\begin{figure}[t]
	\centering
	\subfloat[Tightness]{
		\includegraphics[scale=0.45]{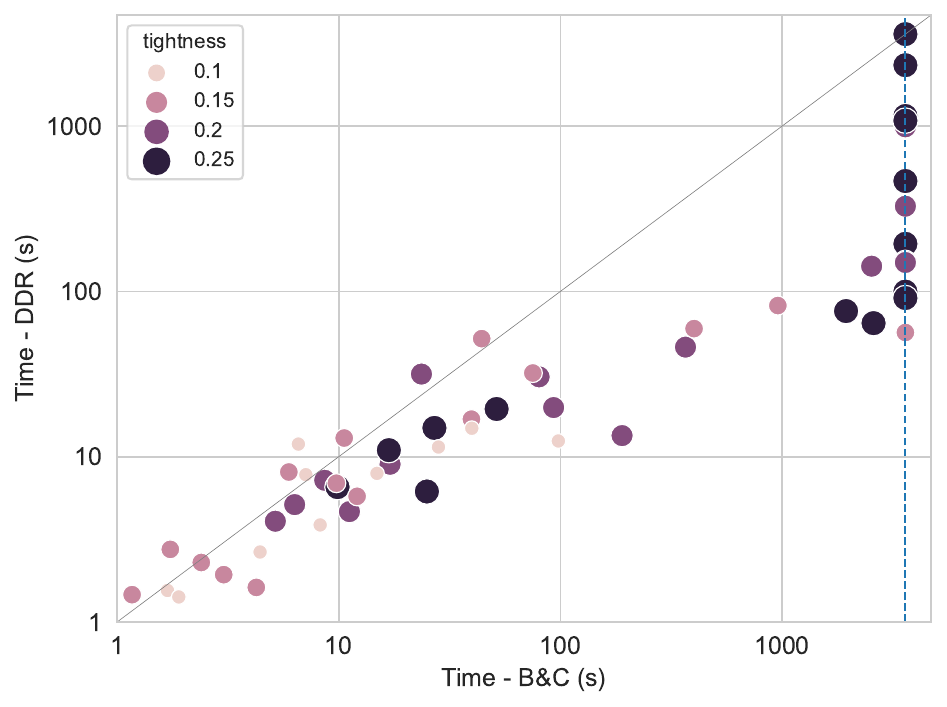}
		\label{fig:tightnessCPSP}
	}%
	\;
	\subfloat[Number of variables]{
		\includegraphics[scale=0.49]{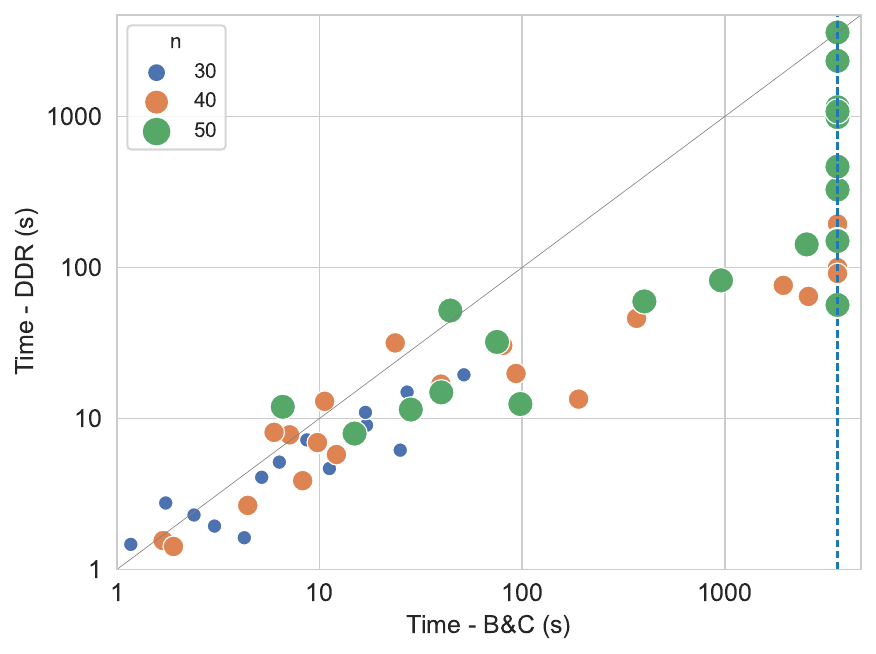}
		\label{fig:nCPSP}
	}%
\caption{(Coloured) Scatter plots comparing runtimes between B\&C and DDR based on the constraint tightness and the number of variables. Dashed lines (in blue) represent the time-limit mark at 3,600 seconds. Both horizontal and vertical coordinates are in logarithmic scale.}
\label{fig:scatterTimesSmall}
\end{figure}

Figure \ref{fig:scatterTimesSmall} depict runtimes through scatter plots that highlight instance tightness (Figure \ref{fig:tightnessCPSP}) and the number of variables $n$ (Figure \ref{fig:nCPSP}). Dashed lines (blue) mark the time-limit coordinate at 3,600 seconds. Runtimes for DDR also include the DD construction. In total, B\&C and DDR solve 47 and 59 instances out of 60, respectively. The runtime for instances solved by B\&C is on average 208 seconds with a high variance (standard deviation of 592.5 seconds), while the runtime for DDR for the same instances is 18 seconds with a standard deviation of 27.5 seconds. 

Figure \ref{fig:tightnessCPSP} suggests that the performance is equivalent for small tightness values (0.1 and 0.15), which represent settings where only few solutions are feasible. However, as tightness increases (0.2 and 0.25), the problem becomes significantly more difficult for both solvers. In particular, while the size of DDR also increases with $t$ (since DDs would represent a larger number of solutions), the model still scales more effectively; instances with tightness of 0.20 and 0.25 are solved in 260.1 seconds on average by DDR. This is at least one order of magnitude faster than B\&C, which could not solve the majority of instances with $t\geq 0.2$ within the time limit. We note that an analogous analysis follows for Figure \ref{fig:nCPSP}, as larger values of $n$ increase the difficulty of the problem for both methods and the same scaling benefits are observed.

Based on these results, we generate two new classes of challenging instances to assess at which point DDR would still be sufficiently compact to provide benefits with respect to B\&C. To this end, the first and second class draw the knapsack coefficients uniformly at random from the discrete uniform distributions $\textnormal{U}(1,50)$ and $\textnormal{U}(1,100)$, respectively, as opposed to $\textnormal{U}(1,25)$. Larger domain distributions significantly increase the DD sizes, since the number of nodes per layers is bounded above by such numbers (we refer to Appendix \ref{sec:Tables} for the final DD sizes). We also vary the tightness within the set $\{0.1, \dots, 0.9\}$ which also impacts the total number of nodes as discussed above. The number of variables is fixed to $n=30$, generating in total 45 instances for each new class with five random instances per configuration.

Figures \ref{fig:performanceU50} and \ref{fig:performanceU100} depict the aggregated results for distributions $\textnormal{U}(1,50)$ and $\textnormal{U}(1,100)$, respectively. Each plot is a performance profile that measures the number of instances solved up to each runtime. We observe that, while instances are more challenging due to larger DD sizes, DDR still solves the 45 problems within the time limit for $\textnormal{U}(1,50)$ and is significantly faster (often by orders of magnitude) than B\&C. For $\textnormal{U}(1,100)$, B\&C and DDR are comparable, indicating the DD size threshold at which DDRs are still beneficial for this problem class. Surprisingly, we observe that for parameters where DDR could potentially suffer from scalability issues (i.e., larger tightness and distribution domains), the resulting problem is also challenging to B\&C, even with a relatively small number of variables $n$. 
\begin{figure}[H]
	\centering
	\subfloat[U(1,50)]{
		\includegraphics[scale=0.49]{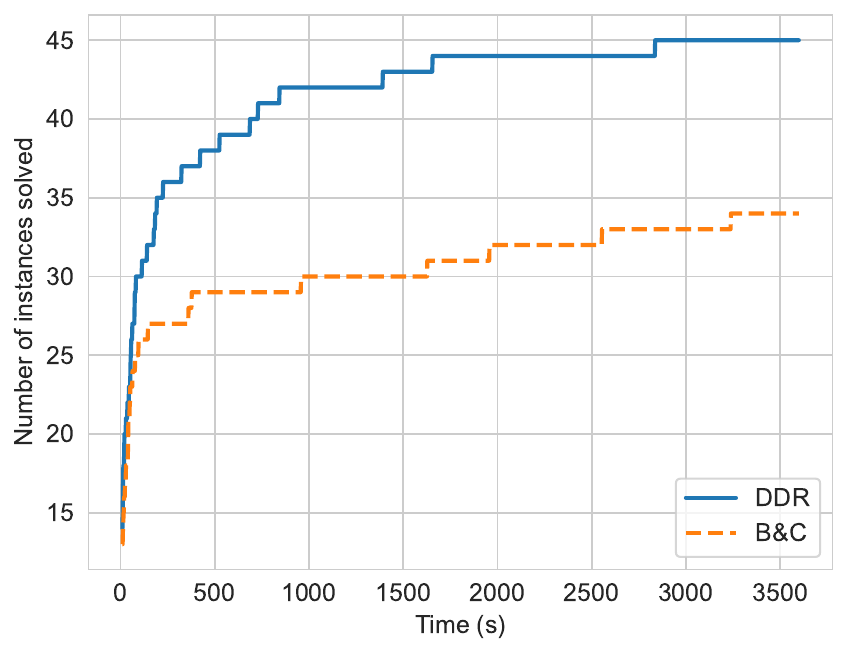}
		\label{fig:performanceU50}
	}%
	\;
	\subfloat[U(1,100)]{
		\includegraphics[scale=0.49]{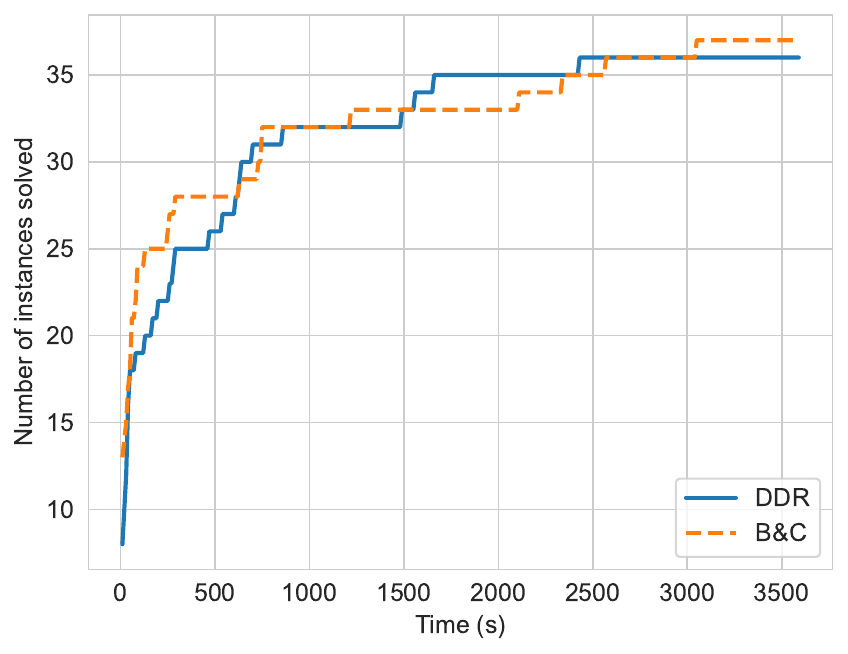}
		\label{fig:performanceU100}
	}%
\caption{(Coloured) Performance profile plot for varying uniform discrete distribution of the knapsack constraint coefficients $a^{F}$ and $a^{L}$.}
\label{fig:performanceLarge}
\end{figure} 

\subsection{Generalizations}
\label{sec:bilevelgeneralizations}
We end this section by discussing extensions of our approach to more general problem settings. Consider a problem setting given by
\begin{subequations} \label{eq:BIPG}
\begin{align}
    \max_{\xvec^L, \xvec^F}
        &\quad
        g^L(\xvec^L, \xvec^F)
        \\
    \textnormal{s.t.}
    &\quad
        H^L(\xvec^L, \xvec^F) \le 0 \;\; 
    \\ \label{BIPG1}
    &\quad
        \xvec^F \in \argmax_{\overline{\xvec}^F} 
        \left\{ 
            g^F(\overline{\xvec}^F)
            \colon            
            H^F(\overline{\xvec}^F) \le 0, \;\;
            \overline{x}^F_i \le 1-\mu_i(\xvec^L), \; \forall i=1,\dots,n , \;\;
            \overline{\xvec}^F \in \{0,1\}^n
        \right\}
        \\ \label{BIPG2}
        &\quad \xvec^L \in \{0,1\}^n,
\end{align}
\end{subequations}
where $g^L$, $H^L:\{0,1\}^{\numvars \times \numvars} \rightarrow \mathbb{R}$ are (possibly nonconvex) functions defined over $\{0,1\}^{\numvars \times \numvars}$; $g^F$, $H^F:\{0,1\}^{\numvars} \rightarrow \mathbb{R}$ are (possibly nonconvex) functions defined over $\{0,1\}^{\numvars}$, and $ \mu_{i}:\{0,1\}^{\numvars} \rightarrow \set{0,1}$ for all $i \in \set{1, \dots, n}$. As before, we assume the relatively complete recourse property and that all functions are well-defined over their domains.   

Consider a decision diagram $\dd^F = (\N,\A,\valvec,\lenvec)$ encoding the discrete feasible set:
\begin{align}
    \feasibleSet = 
    \{
        \xvec \in \{0, 1\}^n 
        \colon
        H^F(\xvec) \le 0
    \} 
\end{align}
and objective $f(\xvec) = g^F(\xvec)$, i.e., paths in $\dd^F$ correspond to follower solutions that satisfy $H^F(\xvec) \le 0$ and path lengths correspond to exact evaluations of function $g^F$. The reformulation procedure and results in Propositions \ref{prop:followerDD} -- \ref{prop:generalBigM} can be extended for the more general setting above by defining the side constraints as
\begin{align}
    y_a \le 1 - \mu_{\layer{a}}(\xvec^L), \;\; \forall a \in \A \colon \val{a} = 1, \label{csp:followerG}
\end{align} 
because the methodology does not require the leader or follower problem to be linear. 
Thus, a single-level reformulation of \eqref{eq:BIPG} is given by
\begin{subequations} \label{BIPGLinear}
\begin{align}
    \label{model:mipbilevelG}
    \tag{MILP-DB}
    \max_{\xvec^L, \xvec^F, \yvec, \pivec, \gammavec} 
        &\quad g^L(\xvec^L, \xvec^F)
        \\
    \textnormal{s.t.}
        &\quad
            H^L(\xvec^L, \xvec^F) \le 0, \\
        &\quad 
             \eqref{follow:ref1.1}-\eqref{follow:ref1.2}, \eqref{follow:ref2}-\eqref{follow:ref2.5}, \eqref{follow:ref4}-\eqref{follow:refVars} \label{mipdp:primalDualG}, \\
        &\quad
           y_a \le 1 - \mu_{\layer{a}}(\xvec^L), \;\; \forall a \in \A \colon \val{a} = 1, \\     
        &\quad 
            \lenvec^{\top} \yvec - \pi_{\rootnode} - \sum_{a \in \A \colon \val{a} = 1 } \gamma_a + \sum_{a \in \A \colon \val{a} = 1 } M_{a} \mu_{\layer{a}}(\xvec^L) = 0,  \label{mipdp:strongdualityG} \\
         &\quad 
            \gamma_a - M_{a} \mu_{\layer{a}}(\xvec^L) \ge 0, &\forall a \in \A \colon \val{a} = 1, \label{mipdp:consistencyG} \\    
        &\quad
            \xvec^L \in \{0,1\}^n.
\end{align}
\end{subequations}
We remark that \eqref{BIPGLinear} is a single-level problem that conserves the original functions $g^L$, $H^L$, and $\mu_{i}$ while replacing functions $g^F$ and $H^F$ with a series of linear constraints based on the variables of a primal flow problem over $\dd^F$ and the corresponding dual variables.


\section{Discrete Robust Optimization}
\label{sec:robustoptimization}

In this section, we introduce a \ref{model:csp} reformulation to robust problems of the form
\begin{align}
\label{model:robust}
\tag{RO}
\min_{\xvec \in \feasibleSet} 
    &\quad
    \obj(\xvec) \\
\textnormal{s.t.}
    &\quad
        \Amat(\delta) \, \xvec \le \bvec(\delta),  &&\forall \delta \in \Uset, \label{robust:subset}
\end{align}
where $\feasibleSet$ is a discrete set as in \ref{model:discretep}, 
$\delta$ are realizations of a (possibly infinite) uncertainty set $\Uset$,
and $\Amat(\delta) \in \mathbb{R}^{r \times n}$, $\bvec(\delta) \in \mathbb{R}^r$ are the coefficient matrix and the right-hand side, respectively, for the realization $\delta \in \Uset$ for some $r > 0$. As discussed in \S\ref{sec:related},
state-of-the-art algorithms for \ref{model:robust} first relax the model by considering only a subset of $\Uset$ in \eqref{robust:subset}, iteratively adding violated constraints from missing realizations until convergence. 

We propose an alternative algorithm for \ref{model:robust} that is applicable when the set $\feasibleSet$ is amenable to a DD encoding. The potential benefit of the methodology is that it is combinatorial and exploits the network structure of $\dd$ for scalability purposes. Our required assumptions are the existence of a separation oracle $\sep(\xvec)$ that identifies violations of \eqref{robust:subset}, i.e.,
\begin{align}
    \sep(\xvec) = 
        \left \{ 
            \begin{array}{ll}
            (i,\delta), & \textnormal{for any $(i,\delta) \in \{1,\dots,r\} \times \Uset$ such that $\avec_i(\delta) \xvec > b_i(\delta)$}, \\
            \emptyset, & \textnormal{if no such pair exists,}
            \end{array}
        \right.
\end{align}
and that such oracle returns either a pair $(i,\delta)$ or $\emptyset$ in finite computational time. This is a mild assumption in existing \ref{model:robust} formulations and holds in typical uncertainty sets, such as when $\Uset$ is an interval domain or, more generally, has a polyhedral description. For instance, if $\Uset$ is finite (e.g., derived from a sampling process), then the simplest $\sep(\xvec)$ enumerates each $\delta$ separately. 

We start in \S\ref{sec:reformulationRobust} with a reformulation of the robust problem as a \ref{model:csp}, leveraging this time the dynamic programming perspective presented in \S\ref{sec:functionalperspective}. Next, we discuss a state-augmenting procedure to address the resulting model in \S\ref{sec:stateAugmenting}, and perform a numerical study on a robust variant of the traveling salesperson problem with time windows (TSPTW) in \S\ref{sec:robustTSPW}. 

\subsection{Reformulation of \ref{model:robust}}
\label{sec:reformulationRobust}
 
Let $\dd^R = (\N,\A,\valvec,\lenvec)$ encode the discrete optimization problem \ref{model:discretep} with feasible set $\feasibleSet$ and objective $\obj(\xvec)$. The constrained longest-path reformulation is obtained directly by representing \eqref{robust:subset} as side constraints over the arc space of $\dd^R$, which we formalize in Proposition \ref{prop:robustReformulation}.

\begin{proposition}
    \label{prop:robustReformulation}
    For each arc $a$ of $\dd^R$, $\delta \in \Uset$, and $i \in \{1,\dots,r\}$, define the scalars 
    \begin{align}
        &g_{i,a}(\delta) = \val{a} \, a_{i, \layer{a}}(\delta),
    \end{align}
    where $a_{i,\layer{a}}(\delta)$ is the $(i, \layer{a})$-th element of $\Amat(\delta)$. Let 
    $\Gmat(\delta) = \{ g_{i,a}(\delta) \}_{\forall i,a}$ be the matrix composed of such scalars for each $\delta$.  There exists a one-to-one mapping between solutions $\xvec$ of \ref{model:robust} and paths of \ref{model:csp} defined over $\dd^R$ and with side constraints
    \begin{align}
        \label{ro:sidecsp}
        \Gmat(\delta) \,\yvec \le \bvec(\delta), && \forall \delta \in \Uset.
    \end{align}
\end{proposition}

Modeling such a reformulation as \ref{model:cspmilp} often results in a binary mathematical program with exponential or potentially infinite many constraints. However, it may be tractable in the presence of special structure. If inequalities \eqref{ro:sidecsp} preserve the totally unimodularity of the network matrix in \eqref{SP1}-\eqref{SP2}, and are separable in polynomial time, then \ref{model:cspmilp} is solvable in polynomial time in $\dd^R$ via the Ellipsoid method \citep{grotschel2012geometric}. Otherwise, computational approaches often rely on decomposition methods that iteratively add variables and constraints to the MILP.

\subsection{State-augmenting Algorithm}
\label{sec:stateAugmenting}

We propose a state-augmenting approach where each iteration is a traditional combinatorial constrained shortest-path problem and is amenable, e.g., to combinatorial methods such as the pulse method \citep{cabrera2020}. Our solution is based on recursive model \ref{model:cspdp} as applied to \ref{model:robust}. More precisely, we solve the recursion 
\begin{align}
    \Vfun_u(\dot \svec) = 
    \left\{
    \begin{array}{ll}
        \min_{a \in \outarcs{u} \colon \svec + \gvec_a \le \hvec} \left \{ \len{a} + V_{\head{a}}(\dot \svec + \gvec_a)  \right\},
            & \textnormal{if $u \neq \terminalnode$,}\\
        0,
            & \textnormal{otherwise.}
    \end{array}    
    \right.
\end{align}
with state space defined by the elements $\dot \svec = ( \; \svec(\delta) \; )_{\forall \delta \in \Uset},$
where each $\svec(\delta) \in \mathbb{R}^r$ is the state vector encoding the $\delta$-th inequalities $\Gmat(\delta) \,\yvec \le \bvec(\delta)$ of \eqref{ro:sidecsp}. 

We recall that, if solved via the shortest-path representation \ref{model:dplp}, an optimal solution corresponds to a sequence $(\dot \svec_1, a_1), \dots (\dot \svec_n, a_n)$ where $(a_1, \dots, a_n)$ is an $\rootnode-\terminalnode$ path in $\dd^R$ and $\dot \svec_i \in \statespace{\tail{a_i}}$ is a reachable state at node $\tail{a_i}$, $i \in \{1,\dots, n\}$. We formalize two properties of the DP model based on this property, which build on the intuition that states conceptually play the role of constraints in recursive models. In particular, we show that only a finite set of realizations are required in the state representation to solve the Bellman equations exactly, providing a ``network counterpart'' of similar results in infinite-dimensional linear programs.

\begin{proposition}
    \label{prop:finiteU}
    Let $z(\Uset')$ be an optimal solution of \ref{model:cspdp} when the state vector $\dot \svec$ is written with respect to a subset of realizations $\Uset' \subseteq \Uset$, adjusting the dimension of $\hvec$ appropriately. Then,
    \begin{enumerate}
        \item[(a)] $z(\Uset') \le z(\Uset'')$ for any $\Uset' \subseteq \Uset'' \subseteq \Uset$.
        \item[(b)] There exists a finite subset $\Uset^* \subseteq \Uset$ such that $z(\Uset^*) = z(\Uset)$.
    \end{enumerate}       
\end{proposition}

\medskip
Proposition \ref{prop:finiteU} and its proof suggest a separation algorithm that identifies a sequence of uncertainty sets $\Uset_1 \subset \Uset_2 \subset \dots \subset \Uset_k = \Uset^*$ providing lower bounds to \ref{model:robust} in each iteration, identifying scenarios to compose $\Delta^*$. We describe it in Algorithm \ref{alg:stateaugmenting}, which resembles a Benders decomposition applied to a recursive model, i.e., the state variables of \ref{model:dplp} encode the Benders cuts and the master problem is solved in a combinatorial way $\dd$ via a constrained shortest-path algorithm. 

The algorithm starts with $\Uset_1 = \emptyset$ and finds an optimal $\rootnode-\terminalnode$ path on $\dd^R$. Notice that such path can be obtained by any shortest-path algorithm because the state space is initially empty. Next, the separation oracle $\sep(\xvec^*)$ is invoked to verify if there exists a realization $\delta$ for which inequalities \eqref{robust:subset} are violated. If that is the case, the state vector $\dot \svec$ is augmented and the problem is resolved with a constrained shortest path algorithm. Otherwise, the algorithm terminates.

\begin{algorithm}[H]
    \begin{algorithmic}[1]
        \State Let $\Uset_1 = \emptyset$, $t = 1$.
        \State Let $\xvec^* = (x_{\val{a_1}}, \dots, x_{\val{a_n}})$ be derived from an optimal $\rootnode-\terminalnode$ path $(a_1,\dots,a_n)$ in \ref{model:dplp} w.r.t. $\Uset_1$.
        \While{there exists $(i,\delta) = \sep(\xvec^*) \neq \emptyset$}
            \State $\Uset_{t+1} = \Uset_t \cup \{ \delta \}$.
            \State Resolve \ref{model:dplp} using the augmented state vector $\dot \svec$ w.r.t. $\Uset_{t+1}$.
            \State Set $\xvec^* = (x_{\val{a_1}}, \dots, x_{\val{a_n}})$ from the optimal $\rootnode-\terminalnode$ path $(a_1, \dots, a_n)$.
            \State $t := t + 1$.
        \EndWhile
        \State \textbf{return} $\xvec^*$.
    \end{algorithmic}
    \caption{State-augmenting Algorithm (Input: \ref{model:robust}, Output: Optimal $\xvec^*$)}
    \label{alg:stateaugmenting}
\end{algorithm}

In each iteration of the algorithm, $ \obj(\xvec^*)$ provides a lower bound on the optimal value of \ref{model:robust} according to Proposition \ref{prop:finiteU}. Once $\sep(\xvec^*) = \emptyset$, we obtain a feasible solution $\xvec^*$ to \ref{model:robust} and, thus, an optimal solution to the problem.    

We provide a formal proof of the convergence of the algorithm below. Specifically, the worst-case complexity of the state-augmenting algorithm is a function of both (i) the number of paths in $\dd^R$, (ii) the worst-case time complexity of the separation oracle $\sep(\xvec)$, and (iii) the worst-case complexity of the solution approach to \ref{model:dplp} to be applied. We observe that, computationally, the choice of each $\delta$ to include plays a key role in numerical performance, which we discuss in our case study of \S \ref{sec:robustTSPW}.

\begin{proposition}
    \label{prop:complexityStateAugmenting}
    Algorithm \ref{alg:stateaugmenting} terminates in a finite number of iterations.  The number of calls to the separation oracle $\sep(\cdot)$ is bounded above by the number of infeasible paths of $\dd^R$ with respect to \eqref{ro:sidecsp}.
\end{proposition}

\subsection{Case Study: Robust Traveling Salesperson Problem with Time Windows}
\label{sec:robustTSPW}

We investigate the separation algorithm on a last-mile delivery problem with uncertain service times and model it as a robust traveling salesperson problem with time windows (RTSPTW). Let $G=(\V,\E)$ be a directed graph with vertex set $\V=\{0,1,\dots,n\}$ and edge set $\E \subseteq \V \times \V$, where $0$ and $n$ are depot vertices. With each vertex $j \in \V \setminus \{0\}$ we associate a time window $[r_j, d_j]$ for a release time $r_i$ and a deadline $d_j$, and with each edge $(i,j)$ we associate non-negative cost $c_{ij}$ and travel time $t_{ij}$. Moreover, visiting a vertex $j \in \V \setminus \{0\}$ incurs a random service time $\delta_j$. The uncertainty values $\uvec = (\delta_1, \dots, \delta_n)$ are described according to a non-empty budgeted uncertainty set 
\begin{align*}
    \Uset = 
    \left\{ 
        \uvec \in \mathbb{Z}^{n} \colon \sum_{j \in \N} \delta_j \le b, \; \; l_j \le \delta_j \le u_j \;\; \forall j \in \N
    \right\},
\end{align*}
where $l_j, u_j \ge 0$ are lower and upper bounds on the service time of vertex $j$, respectively, and $b$ is an uncertainty budget that controls how risk-averse the decision maker is \citep{BertsimasSim04}. More precisely, lower values of $b$ indicate that scenarios for which all vertices of $\V$ will have high service times are unlikely and do not need to be hedged against. Converserly, larger values of $b$ imply that the decision maker is more conversative in terms of the worst case; in particular, $b \rightarrow +\infty$ represents a classical interval uncertainty set.
 
The objective of the RTSPTW is to find a minimum-cost route (i.e., a Hamiltonian path) starting at vertex $0$ and ending at vertex $n$ that observes time-window constraints with respect to the edge travel times. We formalize it as the MILP
\begin{subequations}
\begin{align}
    \min_{\xvec, \wvec} 
        &\quad 
        \sum_{(i,j) \in \E} c_{ij} x_{ij} \label{model:rtstpw} \tag{MILP-TSP} \\
    \textnormal{s.t.}
        &\quad 
            \sum_{j:(i,j)\in \E}{x_{ij}} = 1, 
            &&\forall i \in \V \setminus\{n\}, \label{tsp1}
            \\
        &\quad
            \sum_{j:(j,i)\in \E}{x_{ji}} = 1, 
            &&\forall i \in \V \setminus\{0\}, \label{tsp2} 
            \\
        &\quad 
            w_j^{\uvec} \geq w_i^{\uvec} + (\delta_i+t_{ij}) x_{ij}-M(1-x_{ij}),
            &&\forall (i,j) \in \E, \uvec \in \Uset, \label{tsp3}
            \\
        &\quad 
            r_j \leq  w_j^{\uvec} \leq d_j, 
            &&\forall j \in \V, \del \in \U, \label{tsp4} \\
        &\quad 
            x_{ij} \in \{0,1\}, 
            &&\forall (i,j) \in \E. \label{varsTSP}
\end{align}
\end{subequations}
In the formulation above, the binary variable $x_{ij}$ denotes if the path includes edge $(i,j) \in \E$ and $M=d_n$ is a valid upper bound on the end time of any service. The variable $w^{\uvec}_j$ is the time that vertex $j \in \V$ is reached when the realization of the service times is $\uvec \in \Uset$. Inequalities \eqref{tsp1}-\eqref{tsp2} ensure that each vertex is visited once. Inequalities \eqref{tsp3}-\eqref{tsp4} represent the robust constraint and impose that time windows are observed in each realization $\uvec \in \Uset$. Note that, since $\Uset \neq \emptyset$, they also rule out cycles in a feasible solution. 

The main challenge in \ref{model:rtstpw} is the requirement that routes must remain feasible for all realizations of the uncertainty set. In particular, solving \ref{model:rtstpw} directly is not viable because of the exponential number of variables and constraints. Thus, existing routing solutions are not trivially applicable to this model, as analogously observed in related robust variants of the traveling salesperson problem (e.g., \citealt{montemanni2007robust,bartolini2021robust}). 

\subsubsection{Constrained Shortest-path Formulation.}

We formulate a multivalued decision diagram $\dd$ that encodes the Hamiltonian paths in $G$ starting at $0$ and ending at $n$. Each $\rootnode-\terminalnode$ path $(a_1, \dots, a_n)$ represents the route $(0,\val{a_1}, \dots, \val{a_n})$ with $\val{a_n} = n$ and $\val{a_j} \in \{1,\dots,n-1\}$ for $j \neq n$. Moreover, for simplicity of exposition, we consider that $\dd$ represents directly a traditional DP formulation of the traveling salesperson problem (TSP), i.e., with each node $u \in \N$ we associate the state $(\mathcal{S}_u,L_u)$ where $\mathcal{S}_u$ contains the set of visited vertices on the paths ending at $u$, and $L_u \in \mathcal{V}$ gives the last vertex visited in all such paths. In particular, $(\mathcal{S}_u, L_{\rootnode}) = (\{0\},0)$ and $L_{\terminalnode} = (\mathcal{V}, n)$. We refer to Example \ref{ex:rptswp} for an instance with five vertices and to \cite{cire2013multivalued} for additional details and compression techniques for $\dd$.

\begin{example}
    \label{ex:rptswp}
    Figure \ref{fig:tspDecisionDiagrams} depicts $\dd$ for a TSP with vertex set  $\V = \{0,\dots,4\}$ where every path must start at $0$ and end at $4$. The labels on each arc denote $\val{a}$. For example, in the left-most node of the third layer, the state $(\mathcal{S}_u,L_u) = (\{0,1,2\},2)$ indicates that vertices in $\{0,1,2\}$ were visited in all paths ending at that node, and that the last vertex in such paths is $2$. Notice that the only outgoing arc $a$ has value $\val{a} = 3$, since (i) this is the only vertex not yet visited in such path, and (ii) paths cannot finish at vertex $4$. \hfill $\square$

    \begin{figure}[t]
        \begin{center}
            \tikzstyle{main node} = [circle,fill=gray!50,font=\scriptsize, inner sep=1pt]            
            \tikzstyle{text node} = [font=\scriptsize]
            \tikzstyle{arc text} = [font=\scriptsize]

            \tikzstyle{optimal arc} = [-,line width=2pt, color=teal!20!white]
            \tikzstyle{blocked arc} = [-,line width=2pt, color=red]

            \begin{tikzpicture}[->,>=stealth',shorten >=1pt,auto,node distance=1cm,
                thick, scale=0.85]        
                \tikzstyle{node label} = [fill=gray!20,font=\scriptsize, inner sep=1pt]


                \node[node label] (r) at (-3,0) {$(\{0\}, 0)$};
                \node[node label] (u11)  at (-7, -1.5)  {$(\{0,1\}, 1)$};
                \node[node label] (u12)  at (-3, -1.5)  {$(\{0,2\}, 2)$};
                \node[node label] (u13)  at (1, -1.5)  {$(\{0,3\}, 3)$};
                \node[node label] (u21)  at (-8, -3.25)  {$(\{0,1,2\}, 2)$};
                \node[node label] (u22)  at (-6, -3.25)  {$(\{0,1,3\}, 3)$};                
                \node[node label] (u23)  at (-4, -3.25)  {$(\{0,1,2\}, 1)$};
                \node[node label] (u24)  at (-2, -3.25)  {$(\{0,2,3\}, 3)$};                
                \node[node label] (u25)  at (0, -3.25)  {$(\{0,1,3\}, 1)$};
                \node[node label] (u26)  at (2, -3.25)  {$(\{0,2,3\}, 2)$};
                \node[node label] (u31)  at (-7, -5.5)  {$(\{0,1,2,3\}, 3)$};
                \node[node label] (u32)  at (-3, -5.5)  {$(\{0,1,2,3\}, 2)$};
                \node[node label] (u33)  at (1, -5.5)  {$(\{0,1,2,3\}, 1)$};
                \node[node label] (t)  at (-3, -7)  {$(\{0,1,2,3,4\}, 4)$};
                
                \path[every node/.style={font=\sffamily\small}]
                (r)                 
                     edge[one arc] node [left, arc text, yshift=0.2cm] {1} (u11)
                     edge[one arc] node [arc text] {2} (u12)
                     edge[one arc] node [arc text] {3} (u13)
                (u11) 
                    edge[one arc] node [arc text, left, yshift=0.2cm] {2} (u21)
                    edge[one arc] node [arc text] {3} (u22)                    
                (u12)
                    edge[one arc] node [arc text, left, yshift=0.2cm] {1} (u23)
                    edge[one arc] node [arc text] {3} (u24)                    
                (u13)
                    edge[one arc] node [arc text, left, yshift=0.2cm] {1} (u25)
                    edge[one arc] node [arc text] {2} (u26)
                (u21)
                    edge[one arc] node [arc text, left] {3} (u31)
                (u22)
                    edge[one arc] node [arc text, yshift=-0.15cm] {2} (u32)
                (u23)
                    edge[one arc] node [arc text, left, xshift=-0.1cm] {3} (u31)
                (u24)
                    edge[one arc] node [arc text, yshift=-0.15cm] {1} (u33)
                (u25)
                    edge[one arc] node [arc text,left, xshift=-0.1cm] {2} (u32)                
                (u26)
                    edge[one arc] node [arc text, yshift=0.2cm] {1} (u33)
                (u31)
                    edge[one arc] node [arc text, left, xshift=-0.2cm] {4} (t)
                (u32)
                    edge[one arc] node [arc text] {4} (t)
                (u33)
                    edge[one arc] node [arc text, right, xshift=0.2cm] {4} (t)
                ;
            \end{tikzpicture}        
        \end{center}
        \caption{A decision diagram for a TSP problem with $n=4$ cities for Example \ref{ex:rptswp}.}
        \label{fig:tspDecisionDiagrams}
    \end{figure}
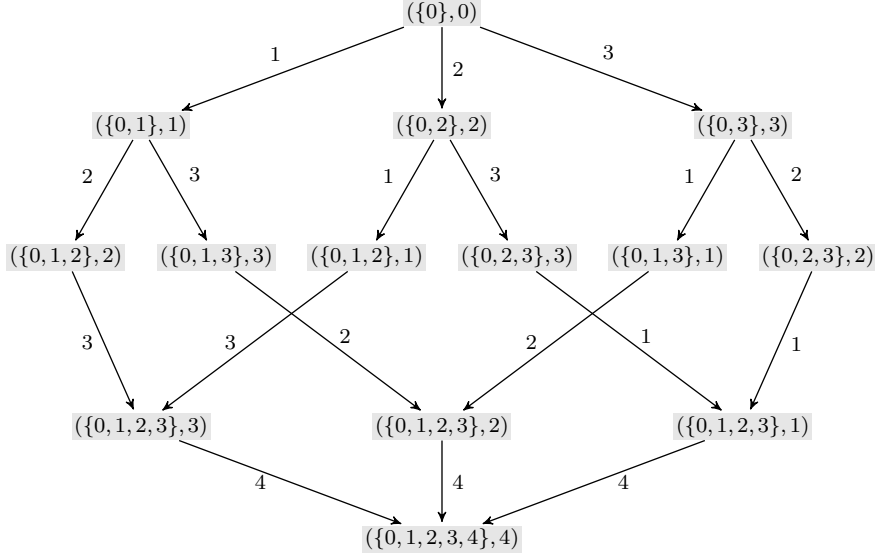

\end{example}

The side constraints are an arc-based representation of \eqref{tsp3}. The inequality system \eqref{csp:side} can be written in terms of arc variables $\yvec$ through auxiliary variables. However, we leverage the DP perspective to obtain a more compact representation. Specifically, for each scenario $\delta \in \Uset$ and node $u \in \N$, let $e_u(\delta)$ be the earliest time any route finishes serving vertex $L_u$ when considering only $\rootnode-\terminalnode$ paths that include $u$. Then,
\begin{align}
    &e_{\rootnode}(\delta) = 0, \\
    &e_{v}(\delta) = \max 
    \left( 
        r_{L_v},
        \min_{(u,v) \in \inarcs{v}} e_u(\delta) + t_{ L_u L_v}
    \right)
    + \delta_{L_v}, && \forall v \in \N \setminus \{ \rootnode \}.
\end{align}
The first equality states that the earliest time we finish serving vertex $0$ is zero by definition. In the second equality, the ``max'' term corresponds to the earliest time to arrive at vertex $L_v$ for a node $v$, which considers its release time and all potential preceding travel time from a predecessor vertex $L_u$. It follows that an $\rootnode-\terminalnode$ path is feasible if and only if
\begin{align}
    e_u(\delta) \le d_{L_u}, && \forall u \in \N, \delta \in \Uset
\end{align}
and the vector $( \, \evec(\delta) \, )_{\delta \in \Uset}$ are the node states augmented by Algorithm \ref{alg:stateaugmenting} per realization $\delta$.

Finally, we require a separation oracle $\sep(\xvec)$ to identify a realization violated by $\xvec$. We propose to pick the ``most violated'' scenario in terms of the number of nodes for which the deadline is not observed (see Appendix \ref{app:sepModel}).

\subsubsection{Numerical Study.} We compare two methodologies for the RTSPTW to evaluate the DD perspective. The first is akin to classical approaches to \ref{model:robust} and consists of applying Algorithm \ref{alg:stateaugmenting} using \ref{model:rtstpw} when deriving a new candidate solution (Step 5); we denote such a procedure by IP. The second approach is Algorithm \ref{alg:stateaugmenting} reformulated via decision diagrams, as developed in this section. We use a simple implementation of the pulse algorithm to solve each iteration of the constraint shortest-path problems over $\dd$ (see description in Appendix \ref{sec:pulse}). We denote our approach by DD-RO. As before, we focus on graphical description of the results, with all associated tables and detail included in Appendix \ref{sec:Tables}. The MILP models were solved in ILOG CPLEX 20.1. Source code will be available at the Github repository. Finally, we remark that both IP and DD-RO can be enhanced using other specialized TSPTW and constrained shortest-path formulations; here we focus on their basis comparison with respect to Algorithm \ref{alg:stateaugmenting}.

Our testbed consists of modified instances for the TSPTW taken from \citet{Dumas95}. In particular, we consider classical problem sizes of $n \in \set{40, 60}$ and time windows widths of $w \in \set{20,40,60,80}$, where a lower value of $w$ indicates tighter time windows and therefore less routing options. For each instance configuration, we incorporate an uncertainty set budget of $\Delta \in \set{4,6,8,10}$ and set the bounds for the service times as $l=0$ and $u=2$. To preserve feasibility we extend the upper time limit for all the nodes by $\Delta$ time units. We generate five random instance per configuration $(n,t,\Delta)$, thus 160 instances in total.

\begin{figure}[h]
	\centering
	\subfloat[Time window width $w$]{
		\includegraphics[scale=0.49]{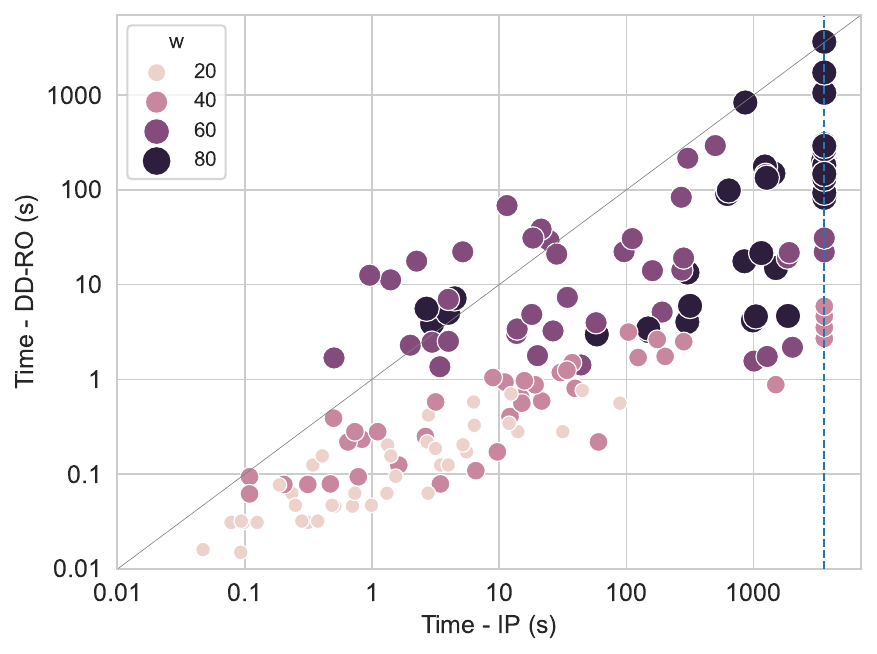}
		\label{fig:widthRTSPTW}
	}%
	\;
	\subfloat[Number of variables]{
		\includegraphics[scale=0.49]{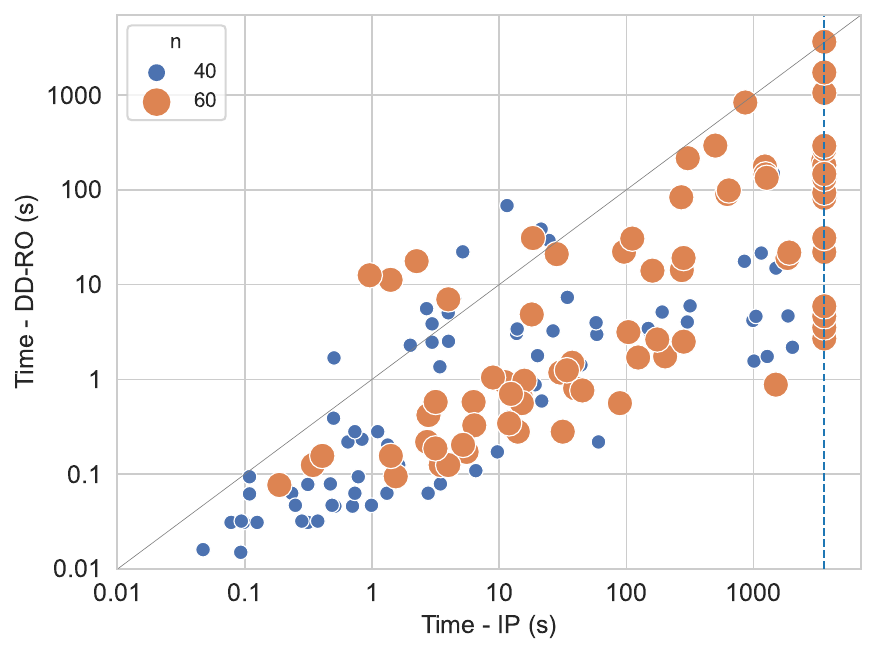}
		\label{fig:nRTSPTW}
	}%
\caption{(Coloured) Scatter plots comparing runtimes between IP and DD-RO based on the time-window width $w$ and the number of variables. Dashed lines (in blue) represent the time-limit mark at 3,600 seconds.}
\label{fig:scatterRTSPTW}
\end{figure} 

Figure \ref{fig:scatterRTSPTW} depict runtimes through scatter plots that compare the time-window width (Figure \ref{fig:widthRTSPTW}) and the number of variables $n$ (Figure \ref{fig:nRTSPTW}), similarly to the previous numerical study in \S\ref{sec:bilevelnumericalanalysis}. Dashed lines (blue) mark the time-limit coordinate at 3,600 seconds. In total, IP and DD-RO solve 137 and 159 instances out of 160, respectively. The runtime for the instances solved by IP is on average 241.9 with a standard deviation of 541.3 seconds, while the runtime for DD-RO for the same instances is 22.1 seconds with a standard deviation of 82.1 seconds. 

Figure \ref{fig:widthRTSPTW} suggests that the complexity of instances is highly affected by time-window width for both methods. In particular, small values of $w$ lead to smaller decision diagrams, and hence faster pulse runtimes. We observe an analogous pattern in Figure \ref{fig:nRTSPTW}, where larger values of $n$ reflect more difficult instances. On average, runtimes for DD-RO were at least 75 times faster than IP, which underestimates the real value since many instances were not solved to optimality by IP. The reason follows from the time per iteration between using pulse and MILP in Algorithm \ref{alg:stateaugmenting}. While both theoretically have the same number of iterations in such instance, pulse is a purely combinatorial approach over $\dd$ and solves the constrained shortest-path problem, on average, in 40 seconds per iteration. For IP, the corresponding model \ref{model:rtstpw} (with fewer scenarios) is challenging to solve, requiring 234 seconds per iteration on average. 

\section{Conclusion}
\label{sec:conclusion}
We study reformulations of a special class of discrete two-stage optimization problems and robust optimization as constrained shortest-path problems (CSP). The methodology consists of reformulating portions of the original problem as a network model, here represented by decision diagrams. The remaining variables and constraints are then incorporated as parameters in budgeted resources over the arcs of the network, reducing the original problem to a CSP.

We propose two approaches to solve the underlying CSP based on each setting. For the first case, we leverage polyhedral structure to rewrite the original problem as a mixed-integer linear programming model, where we convexify non-linear terms connecting the leader's and follower's variables by exploiting duality over the network structure. For the robust case, we proposed a state-augmenting algorithm that iteratively add labels (i.e., resources) to existing dynamic programming approaches to CSPs, which only rely on the existence of a separation oracle that identifies violated realizations of a solution. 

The methods were tested on a competitive project selection problem, where the follower and the leader must satisfy knapsack constraints, and on a robust variant of a traveling salesperson problem with time windows. In both settings, numerical results suggested noticeable improvements in solution time of the CSP-based techniques with respect to state-of-the-art methods on the tested instances, often by orders of magnitude. 

\ACKNOWLEDGMENT{Dr. Lozano gratefully acknowledges the support of the \emph{Office of Naval Research} under grant N00014-19-1-2329, and the \emph{Air Force Office of Scientific Research} under grant FA9550-22-1-0236.}

\bibliographystyle{informs2014}
\bibliography{references.bib}

\clearpage
\begin{APPENDICES}

\section{Proofs}

\begin{proof}{Proof of Proposition \ref{prop:complexity}.}
We reduce the knapsack problem to a formulation of \ref{model:csp} with the required size. Consider any knapsack instance
    \begin{align*}
        \max_{\xvec} 
        \left \{
            p \xvec \colon c \xvec \le b, \;\; \xvec \in \{0,1\}^n
        \right \}
    \end{align*}
    for $a,c \in \mathbb{R}^N$ and $b \in \mathbb{R}^*_{+}$. Consider now a one-width decision diagram having all $n$ permutations of $\{0,1\}$, that is, $\dd$ has one node per layer and every node except $\terminalnode$ has two ougoing arcs, one per each value assignment $\val{a} \in \{0,1\}$. For example, the figure below depicts $\dd$ for $n=4$:
    \medskip
    \begin{center}
        \tikzstyle{main node} = [circle,fill=gray!50,font=\scriptsize, inner sep=1pt]            
        \tikzstyle{text node} = [font=\scriptsize]
        \tikzstyle{arc text} = [font=\scriptsize]

        \tikzstyle{optimal arc} = [-,line width=2pt, color=teal!20!white]
        \tikzstyle{blocked arc} = [-,line width=2pt, color=red]

        \begin{tikzpicture}[->,>=stealth',shorten >=1pt,auto,node distance=1cm,
            thick]        
            \tikzstyle{node label} = [circle,fill=gray!20,font=\scriptsize, inner sep=1pt]

            \node[text node] (l1) at (-1.25, -0.5) {$x_1$};
            \node[text node] (l2) at (-1.25, -1.75) {$x_2$};
            \node[text node] (l4) at (-1.25, -3.15) {$x_3$};
            \node[text node] (l5) at (-1.25, -4.4) {$x_4$};

            \node[node label] (r) at (0,0) {$\;\rootnode\;$};
            \node[node label] (u1)  at (0,-1.25)  {$u_1$};
            \node[node label] (u2)  at (0,-2.5) {$u_2$};
            \node[node label] (u3)  at (0,-3.75) {$u_3$};
            \node[node label] (t) at (0,-5) {$\;\terminalnode\;$};
            
            \path[every node/.style={font=\sffamily\small}]
            (r) 
                edge[zero arc, bend right=45] node [left, arc text] { } (u1)
                edge[one arc, bend left=45] node [right, arc text] { } (u1)
            (u1) 
                edge[zero arc, bend right=45] node [right, arc text] { } (u2)
                edge[one arc, bend left=45] node [right, arc text] { } (u2)
            (u2)
                edge[zero arc, bend right=45] node [left, arc text] { } (u3)
                edge[one arc, bend left=45] node [right, arc text] { } (u3)
            (u3)
                edge[zero arc, bend right=45] node [left, arc text] { } (t)                   
                edge[one arc, bend left=45] node [right, arc text] { } (t);                            
        \end{tikzpicture}        
    \end{center}
    Since $\dd$ encodes the binary set $\{0,1\}^n$, it suffices to define side inequalities \eqref{csp:side} that replicate the knapsack constraint, i.e., we fix $m=1$, coefficients $G_{1 a} = c_{\layer{a}} \val{a}$ for all $a \in \A$, and $\hvec = b$. Finally, for the objective, we have the lengths $\len{a} = -p_{\layer{a}} \val{a}$ for all $a \in \A$, assumed negative since \ref{model:csp} is a ``min'' problem (or, equivalently, we could also change the objective sense to ``max''). Thus, there is a one-to-one correspondance between a knapsack instance and a solution to \ref{model:csp}, and their optimal solution values match. \hfill $\blacksquare$ 
\end{proof}

\medskip
\begin{proof}{Proof of Proposition \ref{prop:DPvalidity}.}
    We will show this result by presenting a bijection between solutions of \ref{model:csp} and $\Vfun_{\rootnode}(\mathbf{0})$, also demonstrating that their solution values match. Specifically, for each node $u \in \N$, let $\pi_u \colon \mathbb{R}^m \rightarrow \A$ be a function that maps a state $s \in \mathbb{R}^m$ to any arc $a \in \A$ that minimizes $\Vfun_{\rootnode}(s)$, i.e., a policy. Since $a \in \outarcs{u}$ in the minimizer of $V_{u}(\cdot)$ for any node $u$, unrolling the recursion $\Vfun_{\rootnode}(\mathbf{0})$ in \ref{model:cspdp} results in a $\rootnode-\terminalnode$ path $(u_1, u_2, u_3, \dots, u_{n+1})$ such that $u_1 = \rootnode$, $u_{n+1} = \terminalnode$, and
    \begin{align*}
        \Vfun_{\rootnode}(\mathbf{0})
        &=
            \len{ \pi_{u_1}(\mathbf{0}) } + V_{\head{\pi_{u_1}(\mathbf{0})}}\left(\mathbf{0} + \gvec_{\pi_{u_1}(\mathbf{0})}\right) \\
        &=
            \len{ \pi_{u_1}(\mathbf{0}) } + \len{ \pi_{u_2}\left(\mathbf{0} + \gvec_{\pi_{u_1}(\mathbf{0})}\right) } + V_{\head{u_1}}\left(\mathbf{0} + \gvec_{\pi_{u_1}(\mathbf{0})} + \gvec_{\pi_{u_2}\left(\mathbf{0} + \gvec_{\pi_{u_1}(\mathbf{0})}\right)}\right) \\
        &= 
            \dots \\            
        &=
            \sum_{j=1}^{n} \len{\pi_{u_j}(\hat{s}_j)}
    \end{align*}
    where $\hat{s}_{1} = \mathbf{0}$ and $\hat{s}_{j} = \hat{s}_{j-1} + G_{\pi_{u_{j-1}}(\hat{s}_{j-1})}$ for $j=2, \dots, n$. Thus, the evaluation of $\Vfun_{\rootnode}(\mathbf{0})$ matches the length of the path $(u_1, u_2, u_3, \dots, u_{n+1})$ considering the arcs specified by the policy $\pi$. Moreover, because of the condition $\svec + \gvec_a \le \hvec$ in the minimizer of $V_{u}(\cdot)$ for any $u \in \N$, we must have $\hat{s}_j \le \hvec$ for all $j = 1,\dots,n$; in particular, for $j=n$ we obtain
    \begin{align*}
        \hvec \ge \hat{s}_n = \mathbf{0} + \sum_{j=1}^{n} G_{\pi_{u_j} (\hat{s}_{j})} = \Gmat \yvec 
    \end{align*}
    for $y_{a} = 1$ if $\pi_{u_j}(\hat{s}_{j-1}) = a$ for some $j \in \{1,\dots,n\}$, and $0$ otherwise. That is, every path in $\dd$ is feasible to \ref{model:csp}. Conversely, every solution to \ref{model:cpsp} can be analogously represented as a policy $\pi$ constructed with the arcs of the associated $\rootnode-\terminalnode$ path, noting that $G$ non-negative implies that every partial sum of arcs also satisfy $\svec + \gvec_a \le \hvec$. \hfill $\blacksquare$
\end{proof}
\medskip

\medskip
\begin{proof}{Proof of Proposition \ref{prop:followerDD}.}
Consider the $\rootnode-\terminalnode$ arc-specified path $p = (a_1, a_2, \dots, a_n)$ in $\dd^F$ extracted from an arbitrary feasible solution $\yvec$ of \ref{model:csp}. The associated solution vector $x^p = (\val{a_1}, \val{a_2}, \dots, \val{a_n})$ satisfies $\Amat^F x^p \le \bvec^F$ by construction of $\dd^F$. Further, because of \eqref{csp:follower}, we can only have $x^p_j = \val{a_j} = 1$ for some $j \in \{1,\dots,n\}$ if $\xvec^{L}_j = 0$, that is, $x^p \le \mathbf{1} - \xvec^L$. Thus, since the solution values and path lengths match by construction of $\dd^F$, an optimal solution of the corresponding \ref{model:csp} satisfies \eqref{bilevel:follower}, and the converse holds analogously. 
    \hfill $\blacksquare$
\end{proof}

\medskip
\begin{proof}{Proof of Proposition \ref{prop:feasiblitySet}.}
    The assumptions of the statement imply that the optimal basic feasible solutions of the linear program
    \begin{align}
        \max_{\yvec}             
            &\quad
                \sum_{a \in \A} \len{a} y_{a} \\
        \text{s.t.} 
            &\quad 
                \sum_{a \in \outarcs{\rootnode}} y_a = 1, \label{eq:lp_ct1} \\
            &\quad 
                \sum_{a \in \Gamma^+(u)} y_a - \sum_{a \in \Gamma^-(u)} y_a = 0, 
                    &&\forall u \in \N \setminus\set{\rootnode, \terminalnode}, \label{eq:lp_ct2} \\
            &\quad 
                \sum_{a \in \A} g_{i,a} y_a \le \hvec_i,  & \forall i \in \{1,\dots,m\} \label{eq:lp_ct3} \\
            &\quad 
                \yvec \in [0,1]^{|\A|}
    \end{align}
    are optimal to \ref{model:cspmilp} and vice-versa, where $g_{i,a}$ is the element at the $i$-th row and $a$-th column of $G$. Consider the dual obtained by associating variables $\pi$ with constraints \eqref{eq:lp_ct1}-\eqref{eq:lp_ct2} and variables $\gamma$ with \eqref{eq:lp_ct3}, i.e.,
    \begin{align}
        \min_{\pi, \gamma}
            &\quad
                \pi_{r} + \sum_{i=1}^m \gamma_i \hvec_i  \\
        \text{s.t.} 
            &\quad 
                \pi_{\tail{a}} - \pi_{\head{a}} + \gammavec^{\top} \gvec_{a} \ge \len{a},
                    &&\forall a \in \A, \\
            &\quad 
                \gamma \ge 0.
    \end{align}
    Then, by strong duality, any optimal solution to both systems satisfy the system $\feasKKT$ composed by the primal and dual feasibility constraints, in addition to inequality \eqref{eq:kkt_ct3} ensuring that the objective function values of the primal and the dual match. Finally, \eqref{eq:kkt_ct4} ensures that $\yvec$ is also integral. \hfill $\blacksquare$
\end{proof}

\medskip
\begin{proof}{Proof of Proposition \ref{prop:optSolFollower}.}
    Consider any arbitrary feasible solution $(\hat{\xvec}^F, \hat{\yvec}, \hat{\pivec}, \hat{\gammavec})$ to \eqref{follow:ref1.1}-\eqref{follow:refVars}. Pick any arc $a \in \A$ such that $x^L_{\layer{a}} = 1$, $\val{a} = 1$, but $\hat{\gamma}_a \neq M_a$. Notice that $\hat{\gamma}_a$ is in constraints \eqref{follow:ref2} and \eqref{follow:ref3}. For \eqref{follow:ref3}, the difference $(1 - x^L_{\layer{a}})$ is zero, and hence changing $\hat{\gamma}_a$ does not impact such an equality. 
    For \eqref{follow:ref2} we restrict our attention to the case $\hat{\gamma}_{a} > M_a$; notice that, otherwise, increasing $\hat{\gamma}_{a}$ would not affect the inequalities. In this case, condition \eqref{bigMCondition} guarantees that \eqref{follow:ref2} is satisfied if we set $\hat{\gamma}_a = M_a$, completing the proof. \hfill $\blacksquare$
\end{proof}

\medskip
\begin{proof}{Proof of Proposition \ref{prop:coolBigM}}
First consider that if $\cost_{j}^F < 0$ then follower variable $x_{j}^F = 0$ at any optimal solution of \ref{model:bilevel}  since $\Amat^F \in \mathbb{R}^{m_F \times n}_+$. Thus we assume without loss of generality that $\cost_{j}^F \geq 0$ for $j \in \set{1, \dots, n}$, which ensures that $\len{a} \geq 0$ for all $a \in \A$.    

We now show that there exists an optimal solution to \ref{model:mipbilevel} that satisfies $\gamma_a = \len{a}$ for all arcs $a \in \A$ such that $x^L_{\layer{a}} = 1$ and $\val{a} = 1$. Consider an optimal solution to \ref{model:mipbilevel} given by $({\xvec^{L*}, \xvec^{F*}, \yvec^*, \pivec^*, \gammavec^*})$ and a decision diagram $\hat{\dd} = (\N,\hat{\A},\valvec,\lenvec)$ that is obtained by removing from $\A$ any arcs for which $x^{L*}_{\layer{a}} = 1$ and $\val{a} = 1$, i.e., removing from $\dd^F$ all the arcs with zero capacity in constraint \eqref{follow:ref1.3}. Formally, the modified set of arcs is given by 
\begin{equation*}
    \hat{\A} = \left\{a \in \A \colon x^{L*}_{\layer{a}} = 0 \text{ or } \val{a} = 0\right\}.
\end{equation*} 
It follows from the non-negativity of $\Amat^F$ that every node in $\dd^F$ has at least one outgoing zero arc and thus the set of nodes of $\hat{\dd}$ is the same as the original set of nodes. 

Since we only remove arcs with zero capacity,  any feasible path in $\hat{\dd}$ is also a feasible path in $\dd^F$, under $\xvec^{L*}$. Thus, the optimal path given by $\yvec^*$ over $\dd^F$ is also an optimal primal solution over $\hat{\dd}$. Next, let us consider an optimal dual solution $\hat{\pivec}$, corresponding to the optimal path given by $\yvec^*$ over $\hat{\dd}$. Given that all the arcs in $\hat{\A}$ have a capacity of $1$, we remove constraints \eqref{follow:ref1.3} from the primal, which in turn removes variables $\gammavec$ from the dual. Dual feasibility of $\hat{\pivec}$ hence ensures that
\begin{equation}
    \hat{\pi}_{\tail{a}} - \hat{\pi}_{\head{a}} \ge \len{a} \quad \forall a \in \hat{\A} \label{dualfeasibilityCool}
\end{equation}
and strong duality ensures that
\begin{equation}
    \lenvec^{\top} \yvec^* - \hat{\pi}_{\rootnode} = 0. \label{dualCool}
\end{equation}
We now show that $ \hat{\pi}_{\tail{a}} - \hat{\pi}_{\head{a}} \geq 0$ for all $a \in \A$. Note that dual variables $\hat{\pivec}$ capture the node potential and are computed as the length of an optimal completion path from any node in $\N$ to $\terminalnode$. These completion paths exists for every node in $\N$ since $\Amat^F$ is non-negative. 

Pick any arc $a \in \A$ such that $\val{a} = 0$. Since $a \in \hat{\A}$ and $\len{a} = 0$, constraints \eqref{dualfeasibilityCool} ensure that $ \hat{\pi}_{\tail{a}} - \hat{\pi}_{\head{a}} \geq 0$. Now pick any arc $a \in \A$ such that $\val{a} = 1$. Assume by contradiction that  $\hat{\pi}_{\tail{a}} - \hat{\pi}_{\head{a}} < 0$. Since $\Amat^F$ is non-negative, there exists an arc $a' \in \hat{\A}$ leaving $\tail{a}$ with $\val{a'} = 0$. Furthermore, observe that any partial solution composed of paths starting from $\head{a}$ and ending at $\terminalnode$ are also found in some other path starting at $\head{a'}$ and ending at $\terminalnode$, again because $A^F$ is non-negative and fixing $x^F_{\layer{a'}}= \val{a'} = 0$ can only increase the number of solutions. As a result we have that $\hat{\pi}_{\head{a'}} \geq \hat{\pi}_{\head{a}}$. Since $a' \in \hat{\A}$ and $\len{a'} = 0$, constraints \eqref{dualfeasibilityCool} ensure that $\hat{\pi}_{\tail{a'}} \geq \hat{\pi}_{\head{a'}}$. Because $\tail{a'} = \tail{a}$ we obtain that $\hat{\pi}_{\tail{a}} \geq \hat{\pi}_{\head{a'}} \geq \hat{\pi}_{\head{a}}$, which contradicts our assumption that $\hat{\pi}_{\tail{a}} - \hat{\pi}_{\head{a}} < 0$.    

Using the results above now consider for any optimal solution to \ref{model:mipbilevel} given by $({\xvec^{L*}, \xvec^{F*}, \yvec^*, \pivec^*, \gammavec^*})$, an alternative solution given by $({\xvec^{L*}, \xvec^{F*}, \yvec^*, \hat{\pivec}, \gammavec'})$, where 
\begin{align*} 
\gamma'_a=&
\begin{cases}
 \len{a} \text{ if }  x^{L*}_{\layer{a}} = 1 \text{ and } \val{a} = 1\\
0 \quad \text{otherwise}
\end{cases} \quad \forall a\in\A.
\end{align*}
Note that $(\hat{\pivec}, \gammavec')$ satisfies constraints \eqref{follow:ref2} and \eqref{follow:ref2.5} because  $ \hat{\pi}_{\tail{a}} - \hat{\pi}_{\head{a}} \geq 0$ for all $a \in \A$. Additionally, $(\yvec^*,\hat{\pivec}, \gammavec')$ satisfies \eqref{follow:ref3} because of \eqref{dualCool} and the fact that $\sum_{a \in \A \colon \val{a} = 1} \gamma_a (1 - x^L_{\layer{a}}) = 0$ by construction. Since the objective value for both solutions is the same, then $({\xvec^{L*}, \xvec^{F*}, \yvec^*, \hat{\pivec}, \gammavec'})$ is an optimal solution to \ref{model:mipbilevel}. \hfill $\blacksquare$
\end{proof}

\medskip
\begin{proof}{Proof of Proposition \ref{prop:generalBigM}}
From Proposition \ref{prop:optSolFollower} we have that valid $M$-values should be large enough to provide an upper bound for the $\gamma$-variables. Consider any feasible solution to \ref{model:mipbilevel} given by $({\hat{\xvec}^{L}, \hat{\xvec}^{F}, \hat{\yvec}, \hat{\pivec}, \hat{\gammavec}})$ and note that for all arcs $a \in \A$ such that $\hat{x}^L_{\layer{a}} = 1$ and $\val{a} = 1$, the corresponding dual $\gamma$-variables could take any value in the range 
\begin{equation*}
    \hat{\gamma}_a \in [ \len{a}+ \hat{\pi}_{\head{a}}-\hat{\pi}_{\tail{a}},\infty],
\end{equation*}
preserving feasibility. Assume without loss of generality that
\begin{equation}
    \hat{\gamma}_a = \len{a}+ \hat{\pi}_{\head{a}}-\hat{\pi}_{\tail{a}},  \quad \forall a \in \A\colon x^L_{\layer{a}} = 1 \text{ and } \val{a} = 1,  \label{gammaDef}
\end{equation}
that is, $\hat{\gamma}_a$ is the shadow price corresponding to constraints \eqref{follow:ref1.3} and captures the change in the objective function associated with adding one unit of capacity to arc $a$. The values for $\hat{\pi}_{\head{a}}$ and $\hat{\pi}_{\tail{a}}$ are computed as the length of the best completion path from nodes $\head{a}$ and $\tail{a}$ to $\terminalnode$, if a completion path exists, or some arbitrary value otherwise. The relatively complete recourse property ensures that there exists at least one path from $\rootnode$ to $\terminalnode$ and as a result
\begin{equation} 
\sum_{j = 1}^n \min\{\cost_{j}^F , 0\} \leq \hat{\pi}_i \leq \sum_{j = 1}^n \max\{\cost_{j}^F , 0\}, \quad \forall i\in\N.    \label{pidef}
\end{equation}
Combining \eqref{gammaDef} and \eqref{pidef} we obtain that
\begin{equation}
    \hat{\gamma}_a \leq \len{a} + \sum_{j = 1}^n \max\{\cost_{j}^F , 0\} - \sum_{j = 1}^n \min\{\cost_{j}^F , 0\},  \quad \forall a \in \A\colon x^L_{\layer{a}} = 1 \text{ and } \val{a} = 1,
\end{equation}
completing the proof. 
 \hfill $\blacksquare$
\end{proof}

\medskip
\begin{proof}{Proof of Proposition \ref{prop:robustReformulation}}
By construction of $\dd^R$, there is a one-to-one correspondence between paths of $\dd^R$ and solutions of $\feasibleSet$. Consider the $\rootnode-\terminalnode$ arc-specified path $p = (a_1, a_2, \dots, a_n)$ in $\dd^R$ extracted from an arbitrary feasible solution $\yvec$ of \ref{model:csp}, and let $x^p = (\val{a_1}, \val{a_2}, \dots, \val{a_n}) \in \feasibleSet$ be its associated solution. Then, for any $\delta \in \Uset$ and $i \in \{1,\dots,r\}$,
    \begin{align*}
        \bvec_i(\delta) \ge g_i(\delta) \, y = \sum_{a \in \A} g_{i,a}(\delta) \, y_a = \sum_{a \in \A} \val{a} \, a_{i, \layer{a}}(\delta) \, y_a
        = \sum_{j=1}^n a_{i, j}(\delta) \val{a_j} = \sum_{j=1}^n a_{i, j}(\delta) x^p_j = a_i x^p.
    \end{align*}
    That is, $x^p$ is feasible to \ref{model:robust}. Conversely, any feasible solution to \ref{model:robust} can be shown feasible to \ref{model:csp} by applying the same steps, noticing that the objective functions are equivalent.
    \hfill $\blacksquare$
\end{proof}

\medskip
\begin{proof}{Proof of Proposition \ref{prop:finiteU}.}
    Property (a) follows because \ref{model:cspdp} is equivalent to \ref{model:robust} for any $\Uset' \subseteq \Uset$ due to Proposition \ref{prop:robustReformulation} (correctness of the reformulation) and Proposition \ref{prop:DPvalidity} (correctness of \ref{model:cspdp}). Thus, since \ref{model:robust} written in terms of $\Uset'$ has less constraints than the same model written in terms of $\Uset''$, we must have $z(\Uset') \le z(\Uset'')$.
    
    For property (b), let $\mathcal{P}$ be the finite set of $\rootnode-\terminalnode$ paths in $\dd^R$. Partition $\mathcal{P}$ into subsets $\mathcal{P}^F$ and $\mathcal{P}^{I} = \mathcal{P} \setminus \mathcal{P}^F$, where $p \in \mathcal{P}^F$ if and only if its associated solution $x^p$ is feasible to \ref{model:robust}. Thus, for each $\hat{p} \in \mathcal{P}^I$, there exists at least one scenario $\delta^{\hat{p}} \in \Uset$ where some inequality of $A(\delta^{\hat{p}}) \le b(\delta^{\hat{p}})$ is violated. We build the set
    \begin{align*}
        \Uset^* = \bigcup_{\hat{p} \in \mathcal{P}^I} \{ \delta^{\hat{p}} \}
    \end{align*}
    which is also finite and must be such that $z(\Uset^*) = z(\Uset)$, since only paths in $\mathcal{P}^F$ are feasible to \ref{model:csp} which, by definition, are also feasible to \ref{model:robust}. \hfill $\blacksquare$
\end{proof}

\medskip
\begin{proof}{Proof of Proposition \ref{prop:complexityStateAugmenting}}
At any iteration $t$, the algorithm either (i) stops, in which case the solution is optimal since $\Uset_{t+1} \subset \Uset$; or (ii) $\xvec^*$ is infeasible due to some scenario $\delta$. Because $\delta$ belongs to the future sets $\Uset_{t+2}, \Uset_{t+3}, \dots$, we must have that $\xvec^*$ is infeasible to \ref{model:dplp} at any future iterations, i.e., its associated path is never selected as an optimal $\rootnode-\terminalnode$ path. Thus, since the number of paths of $\dd^R$ is finite, the algorithm must terminate after all infeasible paths have been exhausted. \hfill $\blacksquare$
\end{proof}

\bigskip

\section{Decision Diagram Modeling}
\label{sec:modelDD}

In this section, we briefly discuss the two general steps of modeling a problem as a discrete optimization problem \ref{model:discretep}; additional details and examples are provided in \cite{bergman2016}. Specifically, the first step rewrites the problem as the recursive model
\begin{align*}
    \min_{\xvec, \svec} 
        &\quad 
            \sum_{i=1}^n \hat{f}(s_i, x_i) \\
    \textnormal{s.t.} 
        &\quad 
            s_{i+1} = T_i(s_{i}, x_{i}), && \forall i \in \{1,\dots,n\}, \\
        &\quad 
            x_{i} \in \mathcal{X}(s_{i}), && \forall i \in \{1,\dots,n\},
\end{align*}
where $s_1, \dots, s_{n}, s_{n+1}$ are state variables, $T_i(s_i, x_i)$ for all $i = 1,\dots,n$ are state-transition functions, $\hat{f}(s_i, x_i)$ are cost functions, and $\mathcal{X}(s_{i})$ are the possible values that variable $x_i$ admits given the state element $s_i$. A recursive reformulation is valid when the feasible vectors $\xvec$ of the system above have a one-to-one relationship with $\feasibleSet$ and 
$\sum_{i=1}^n \hat{f}(s_i, x_i) = f(\xvec)$. For instance, for the knapsack problem 
\begin{align*}
    \max_{\xvec} 
        &\quad
            \sum_{j=1}^n c_j x_j  \\
    \textnormal{s.t.}
        &\quad 
            \sum_{j=1}^n a_j x_j \le b \\
        &\quad
            x_j \in \{0,1\}, &&\forall j \in \{1,\dots,n\},
\end{align*}
with $a$ and $c$ non-negative, the recursive model is 
\begin{align*}
    \max_{\xvec, \svec} 
        &\quad
            \sum_{j=1}^n c_j x_j  \\
    \textnormal{s.t.}
        &\quad 
            s_{j+1} = s_{j} + a_j x_j, && \forall j \in \{1,\dots,n\} \\
        &\quad
            s_{j} + a_j x_j \le b, &&\forall j \in \{1,\dots,n\}, \\
        &\quad 
            x_j \in \{0,1\}, &&\forall j \in \{1,\dots,n\}, \\
        &\quad 
            s_1 = 0
\end{align*}
that is, $\svec_j$ encodes the knapsack load at the $j$-th stage, $T_i(s_j, x_j) = s_j + a_j x_j$, $\hat{f}(s_j,x_j) = c_j x_j$, and $\feasibleSet(s_j) = \{ x \in \{0,1\} \colon s_j + a_j x_j \le b\}$. 

Given the recursive formulation above, we write the initial $\dd$ as the underlying state-transition graph, i.e., the $j$-th layer $\N_j$ includes one node per possible value of $s_j$, and there exists an arc with label $x_j$ between two nodes $s_j$, $s_{j+1}$ if and only if $s_{j+1} = s_{j} + a_j x_j$ and $x_j \in \feasibleSet(s_j)$. The length of such an arc is set to $\hat{f}(s_i, x_i)$. For example, consider the knapsack instance from Example \ref{ex:dd}:
\begin{align*}
    \max_{\xvec} 
        \quad&  
            4x_1 + 3x_2 + 7x_3 + 8x_4 \\ 
    \textnormal{s.t.}
        \quad&
            7 x_1 + 5 x_2 + 4 x_3 + x_4 \leq 8, \\
        \quad&
            x_1, x_2, x_3, x_4 \in \{0,1\}.
\end{align*}
The resulting (non-reduced) $\dd$ is as follows, where node labels correspond to the states $s_j$ and arc labels to $\hat{f}(s_i, x_i)$. We also represent all possible values of $s_{n+1}$ into a single terminal node $\terminalnode$ with no loss of information. For exposition purposes, we also omit the labels on the last arcs, which are all equivalent ($0$ for dashed arcs, and $1$ for solid arcs).

\medskip

\begin{center}
    \tikzstyle{main node} = [circle,fill=gray!50,font=\scriptsize, inner sep=1pt]            
    \tikzstyle{text node} = [font=\scriptsize]
    \tikzstyle{arc text} = [font=\scriptsize]

    \tikzstyle{optimal arc} = [-,line width=2pt, color=teal!20!white]
    \tikzstyle{blocked arc} = [-,line width=2pt, color=red]

    $\quad \; $
    \begin{tikzpicture}[->,>=stealth',shorten >=1pt,auto,node distance=1cm,
        thick]        
        \tikzstyle{node label} = [circle,fill=gray!20,font=\scriptsize, inner sep=1pt]

        \node[node label] (r) at (0,0) {$0$};
        \node[node label] (u1a)  at (-1,-1.25)  {$0$};
        \node[node label] (u1b)  at (1,-1.25)  {$7$};
        \node[node label] (u2a)  at (-1,-2.5) {$0$};
        \node[node label] (u2b)  at (0,-2.5) {$5$};
        \node[node label] (u2c)  at (1,-2.5) {$7$};
        \node[node label] (u3a)  at (-3,-3.75) {$0$};
        \node[node label] (u3b)  at (-1.5,-3.75) {$4$};
        \node[node label] (u3c)  at (0,-3.75) {$5$};
        \node[node label] (u3d)  at (1.5,-3.75) {$7$};
        \node[node label] (t)  at (0,-5) {$\terminalnode$};
        
        \path[every node/.style={font=\sffamily\small}]
        (r)             
            edge[zero arc] node [left, arc text] {$0$} (u1a)
            edge[one arc] node [right, arc text] {$4$} (u1b)
        (u1a)
            edge[zero arc] node [right, arc text] {$0$} (u2a)
            edge[one arc] node [right, arc text] {$\;\;3$} (u2b)
        (u1b)
            edge[zero arc] node [left, arc text] {$0$} (u2c)
        (u2a)
            edge[zero arc] node [left, arc text] {$0$} (u3a)
            edge[one arc] node [right, arc text] {$7$} (u3b)
        (u2b)
            edge[zero arc] node [left, arc text] {$0$} (u3c)
        (u2c)
            edge[zero arc] node [left, arc text] {$0$} (u3d)        
        (u3a)
            edge[zero arc, bend right=30] node [left, arc text] {} (t)
            edge[one arc, bend right=10] node [right, arc text] {} (t)
        (u3b)
            edge[zero arc, bend right=20] node [left, arc text] {} (t)
            edge[one arc, bend left=20] node [left, arc text] {} (t)
        (u3c)
            edge[zero arc, bend right=20] node [left, arc text] {} (t)
            edge[one arc, bend left=20] node [left, arc text] {} (t)
        (u3d)
            edge[zero arc, bend right=20] node [left, arc text] {} (t)
            edge[one arc, bend left=20] node [left, arc text] {} (t)
            ;
    \end{tikzpicture}        
\end{center}

Finally, the second step of the procedure is known as reduction, where we compress the graph by eliminating redundant nodes. For example, all nodes in the fourth layer can be compressed into a single node, since they have equivalent paths to $\terminalnode$. More precisely, two nodes $u, u' \in \N_j$ are redundant if their outgoing arcs are equivalent, that is, for each arc $a$ outgoing from $u$, there exists another arc $a'$ with the same head node, length, and value; and vice-versa for node $u'$. The reduction procedure is a bottom-up traversal that merges nodes if their outgoing arcs are equivalent, i.e.,
\begin{enumerate}
     \item \textbf{For each} layer $j = n, n-1, n-2, \dots, 2$:
     \begin{enumerate}
         \item \textbf{If} there are two redundant nodes $u, u' \in \N_j$, \textbf{then} merge them into a single node, redirecting the corresponding incoming arcs.
     \end{enumerate}
\end{enumerate}

Figure \ref{fig:knap_dd}(a) depictes the resulting graph after applying the reduction mechanism above.

\section{The Pulse Algorithm}
\label{sec:pulse}
The pulse algorithm proposed originally by \citet{lozano2013}, is a recursive enumerative search based on the idea of propagating a \emph{pulse} through a directed network, starting at the root node. Pulses represent partial paths and propagate through the outgoing arcs from each node, storing information about the partial path being explored. Once a pulse corresponding to a feasible path reaches the end node, the algorithm updates a primal bound on the objective function, stops the pulse propagation, and backtracks to continue the recursive search resulting in a pure depth-first exploration. Opposite to labelling algorithms, pulses are not stored in memory but passed as arguments in the recursive search function. If nothing prevents pulses from propagating, the pulse algorithm enumerates all feasible paths from root node to end node ensuring that an optimal path is found. 

To avoid a complete enumeration of the solution space, the pulse algorithm employs a set of pruning strategies proposed to prune pulses corresponding to partial paths, without cutting off an optimal solution. The core strategies in the pulse algorithm are pruning by infeasibility, bounds, and dominance. The infeasibility pruning strategy uses a dual bound on the minimum resource required to reach the end node from any given intermediate node and discards partial paths that cannot be completed into a feasible complete path. The bounds pruning strategy uses dual bounds on the minimum-cost (maximum-profit) paths from any intermediate node to the end node and discards partial paths that cannot be part of an optimal solution. The dominance pruning strategy uses a limited number of labels to store information on partial paths explored during the search and discard partial paths based on traditional dominance relationships. In contrast to labeling algorithms, the pulse algorithm does not rely on extending every non-dominated label for correctness; instead, this is guaranteed by properly truncating the recursive search. Hence, even if no labels were used at any of the nodes, the algorithm remains correct.

The pulse algorithm has been extended for the elementary shortest path problem with resource constraints \citep{Lozano2015}, the biobjective shortest path problem \citep{Duque2015}, the weight constrained shortest path problem with replenishment \citep{Bolivar2014}, the orienteering problem with time windows \citep{Duque2014}, and more recently, the robust shortest path problem \citep{Duque2019}. Beyond the domain of shortest path problems, several authors have used the pulse algorithm as a component to solve other hard combinatorial problems. For instance, the pulse algorithm has been used in network interdiction \citep{Lozano2017}, shift scheduling \citep{Restrepo2012}, evasive flow \citep{Arslan2018}, resource constrained pickup-and-delivery \citep{Schorotenboer2019}, and green vehicle routing problems \citep{Montoya2016}, among others.

\section{Separating the Most Violated Scenario for RTSPTW} \label{app:sepModel}
For ease of notation, let $r_0 = 0$ and $d_0 = +\infty$. To separate the scenario in which the deadlines are most violated we solve model \ref{model:sep} for a fixed $\xvec$:
\begin{align}
    \max_{\vvec, \delta, \wvec} 
        &\quad 
            \sum_{j \in \V} v_j \label{model:sep} \tag{SEP-TSP} \\
    \textnormal{s.t.}
        &\quad
            w_0 = 0, \label{sep:ct1} \\
        &\quad 
            w_j = \max \{ r_{j}, w_i + \delta_i + t_{ij} \},
            &&\forall (i,j) \in \E \colon x_{ij} = 1, \label{sep:ct2}
            \\
        &\quad 
            w_j \ge (d_j + \epsilon)v_j,
            && \forall j \in \V, \label{sep:ct4} \\
        &\quad 
            \delta \in \Uset, v_j \in \{0,1\}, 
            && \forall j \in \V. \label{sep:ct5}
        \end{align}
In our separation model, $v_j = 1$ if and only if the time window of vertex $j$ is violated, $j \in \V$. The equalities \eqref{sep:ct1}-\eqref{sep:ct2} specify the arrival times $\wvec$ that follow from $\xvec$. The inequalities \eqref{sep:ct4} enforce the definition of $v_j$ considering a sufficiently small $\epsilon$ (e.g., the machine precision). Finally, \eqref{sep:ct5} specify the domain of variables. We note that \ref{model:sep} can be solved by MILP solvers after linearizing the ``max'' in \eqref{sep:ct2} with additional binary variables.

\section{Additional Tables}
\label{sec:Tables}
\subsection{Competitive Project Scheduling}
Table \ref{tab:CPS} presents extended results for our experiments with competitive scheduling instances. The first column indicates the number of items. The second column presents the budget tightness parameter. Columns 3 to 8 present for each algorithm the average computational time in seconds calculated among the 5 instances of the same configuration, the number of instances solved within the time limit, and the average optimality gap computed over the instances not solved within the time limit. We use a dash in the ``Gap" column to indicate that all instances in the row are solved to optimality. For instances not solved within the time limit we record a computational time of $3600$ seconds. Columns 9 and 10 present for the DD-based approach the average number of nodes and arcs in the resulting DD. 
\begin{table}[H]
\footnotesize
\caption{Comparing \texttt{B\&C} and \texttt{DDR} on a set of synthetic competitive scheduling problem instances }
\label{tab:CPS}
\centering
\begin{tabular}{rcrrrrrrrr}
\hline
\multicolumn{ 1}{c}{{\it {}}} &  \multicolumn{ 3}{c}{{\bf \texttt{B\&C}}} & \multicolumn{ 5}{c}{{\bf \texttt{DDR}}} \\
   { $n$} & {\bf $\omega$} & {\bf Time (s)} & {\bf $\#$ Solved} &  {\bf Gap} & {\bf Time (s)} & {\bf $\#$ Solved} &  {\bf Gap} & {\bf Nodes} & {\bf Arcs} \\
\hline
           &        0.1 &          1 &          5 &          - &          1 &          5 &          - &        403 &        783 \\

           &       0.15 &          3 &          5 &          - &          2 &          5 &          - &        675 &       1328 \\

        30 &        0.2 &         10 &          5 &          - &          6 &          5 &          - &        921 &       1821 \\

           &       0.25 &         26 &          5 &          - &         12 &          5 &          - &       1126 &       2234 \\

           &            &            &            &            &            &            &            &            &            \\

           &        0.1 &          5 &          5 &          - &          3 &          5 &          - &        907 &       1784 \\

           &       0.15 &         16 &          5 &          - &         10 &          5 &          - &       1407 &       2786 \\

        40 &        0.2 &        151 &          5 &          - &         28 &          5 &          - &       1847 &       3669 \\

           &       0.25 &       3066 &          2 &        3\% &        105 &          5 &          - &       2232 &       4440 \\

           &            &            &            &            &            &            &            &            &            \\

           &        0.1 &         37 &          5 &          - &         12 &          5 &          - &       1531 &       3025 \\

           &       0.15 &       1016 &          4 &      0.4\% &         56 &          5 &          - &       2365 &       4695 \\

        50 &        0.2 &       3387 &          1 &        7\% &        352 &          5 &          - &       3088 &       6145 \\

           &       0.25 &   $>$3600 &          0 &       11\% &       1727 &          4 &        3\% &       3706 &       7383 \\
\hline
           & {\bf Average} &        943 &         47 &        7\% &        193 &         59 &        3\% &       1684 &       3341 \\
\hline
\end{tabular}  
\end{table}

Table \ref{tab:CPS2} presents additional results obtained by fixing the number of projects to $n=30$ and exploring different parameters for the distribution of the projects cost as well as a wider range of the budget tightness parameter. 

\begin{table}[H]
\footnotesize
\caption{Sensitivity analysis for \texttt{B\&C} and \texttt{DDR} on a set of synthetic competitive scheduling problem instances }
\label{tab:CPS2}
\centering
\begin{tabular}{rcrrrrrrrr}
\hline
\multicolumn{ 1}{c}{{\it {}}} &  \multicolumn{ 3}{c}{{\bf \texttt{B\&C}}} & \multicolumn{ 5}{c}{{\bf \texttt{DDR}}} \\
   {\bf Distribution} & { $\omega$} & {\bf Time (s)} & {\bf $\#$ Solved} &  {\bf Gap} & {\bf Time (s)} & {\bf $\#$ Solved} &  {\bf Gap} & {\bf Nodes} & {\bf Arcs} \\
\hline
           &        0.1 &          2 &          5 &          - &          1 &          5 &          - &        569 &       1116 \\

           &        0.2 &         14 &          5 &          - &         13 &          5 &          - &       1422 &       2826 \\

           &        0.3 &         50 &          5 &          - &         96 &          5 &          - &       2051 &       4086 \\

           &        0.4 &        497 &          5 &          - &        478 &          5 &          - &       2435 &       4857 \\

   U[1,50] &        0.5 &       $>$3600 &          0 &       16\% &       1209 &          5 &          - &       2565 &       5120 \\

           &        0.6 &       $>$3600 &          0 &        8\% &        371 &          5 &          - &       2437 &       4865 \\

           &        0.7 &       2395 &          4 &        6\% &         45 &          5 &          - &       2052 &       4097 \\

           &        0.8 &        122 &          5 &          - &          9 &          5 &          - &       1423 &       2841 \\

           &        0.9 &          1 &          5 &          - &          2 &          5 &          - &        579 &       1153 \\

           &            &            &            &            &            &            &            &            &            \\

           &        0.1 &          2 &          5 &          - &          2 &          5 &          - &        724 &       1425 \\

           &        0.2 &         19 &          5 &          - &         53 &          5 &          - &       2236 &       4452 \\

           &        0.3 &         98 &          5 &          - &        596 &          5 &          - &       3500 &       6983 \\

           &        0.4 &        730 &          5 &          - &       2622 &          2 &        2\% &       4282 &       8551 \\

  U[1,100] &        0.5 &       2804 &          3 &       10\% &       3023 &          1 &        7\% &       4536 &       9060 \\

           &        0.6 &       $>$3600 &          0 &        4\% &       2332 &          3 &        2\% &       4284 &       8559 \\

           &        0.7 &       1562 &          4 &      0.2\% &        459 &          5 &          - &       3505 &       7002 \\

           &        0.8 &         72 &          5 &          - &         39 &          5 &          - &       2243 &       4481 \\

           &        0.9 &          2 &          5 &          - &          8 &          5 &          - &        735 &       1466 \\
\hline
           & {\bf Average} &       1065 &         71 &        9\% &        631 &         81 &        4\% &       2310 &       4608 \\
\hline
\end{tabular}  
\end{table}

\subsection{Robust Traveling Salesman Problem}
Table \ref{tab:TSP} presents additional results for our experiments with robust traveling salesman instances. The first column indicates the number of nodes. The second column presents the time windows width parameter and the third column shows the budget for the uncertainty set. Columns 4 to 9 present for each algorithm the average computational time in seconds calculated among the 5 instances of the same configuration, the number of instances solved within the time limit, and the average optimality gap.
\begin{table}[h]
\footnotesize
\caption{Comparing \texttt{IP} and \texttt{DD-RO} on a set of robust traveling salesman problem instances}
\label{tab:TSP}
\centering
\begin{tabular}{rrrrrrrrr}
\hline
    {\bf } &     {\bf } &     {\bf } &        \multicolumn{ 3}{c}{{\bf \texttt{IP}}} &     \multicolumn{ 3}{c}{{\bf \texttt{DD-RO}}} \\

   {\bf n} &    {\bf w} & {\bf Budget} & {\bf Time} & {\bf \# Solved} &  {\bf Gap} & {\bf Time} & {\bf \# Solved} &  {\bf Gap} \\
\hline
           &            &          4 &        0.4 &          5 &       - &        0.1 &          5 &       - \\

           &            &          6 &          1 &          5 &       - &        0.1 &          5 &       - \\

        40 &         20 &          8 &          1 &          5 &       - &        0.1 &          5 &       - \\

           &            &         10 &          1 &          5 &       - &        0.1 &          5 &       - \\

           &            &          4 &          1 &          5 &       - &        0.2 &          5 &       - \\

           &            &          6 &          4 &          5 &       - &        0.2 &          5 &       - \\

        40 &         40 &          8 &          7 &          5 &       - &        0.2 &          5 &       - \\

           &            &         10 &         17 &          5 &       - &        0.4 &          5 &       - \\

           &            &          4 &         16 &          5 &       - &          6 &          5 &       - \\

           &            &          6 &        223 &          5 &       - &          8 &          5 &       - \\

        40 &         60 &          8 &        302 &          5 &       - &         10 &          5 &       - \\

           &            &         10 &        417 &          5 &       - &         17 &          5 &       - \\

           &            &          4 &        420 &          5 &       - &         35 &          5 &       - \\

           &            &          6 &       1249 &          4 &       20\% &         45 &          5 &       - \\

        40 &         80 &          8 &       1293 &          4 &       20\% &         49 &          5 &       - \\

           &            &         10 &       1223 &          4 &       20\% &         68 &          5 &       - \\

           &            &          4 &          2 &          5 &       - &        0.2 &          5 &       - \\

           &            &          6 &          6 &          5 &       - &        0.3 &          5 &       - \\

        60 &         20 &          8 &         11 &          5 &       - &        0.3 &          5 &       - \\

           &            &         10 &         31 &          5 &       - &        0.4 &          5 &       - \\

           &            &          4 &        751 &          4 &       20\% &          1 &          5 &       - \\

           &            &          6 &        768 &          4 &       20\% &          2 &          5 &       - \\

        60 &         40 &          8 &        777 &          4 &       20\% &          2 &          5 &       - \\

           &            &         10 &       1104 &          4 &       20\% &          3 &          5 &       -\\

           &            &          4 &        145 &          5 &       - &         27 &          5 &       - \\

           &            &          6 &        682 &          5 &       - &         42 &          5 &       - \\

        60 &         60 &          8 &        466 &          5 &       - &         60 &          5 &       - \\

           &            &         10 &       1563 &          3 &       40\% &         78 &          5 &       -\\

           &            &          4 &       1982 &          3 &       40\% &        266 &          5 &       - \\

           &            &          6 &       2995 &          2 &       60\% &        327 &          5 &       - \\

        60 &         80 &          8 &       3134 &          1 &       80\% &        512 &          5 &       - \\

           &            &         10 &       3600 &          0 &       100\% &        898 &          4 &       20\% \\
\hline
           &            &    Average &        725 &        137 &       38\% &         77 &        159 &       20\% \\
\hline
\end{tabular}  
\end{table}

\end{APPENDICES}

\end{document}